\documentclass[a4paper,12pt]{article}
\usepackage{latexsym}
\usepackage{amsmath}
\usepackage{amssymb}
\usepackage{enumerate}
\usepackage{bigints}
\usepackage{tikz}
\usetikzlibrary{trees}
\usetikzlibrary{calc}
\usepackage{theorem}
\newtheorem{theorem}{Theorem}[section]
\newtheorem{proposition}[theorem]{Proposition}
\newtheorem{lemma}[theorem]{Lemma}
\newtheorem{corollary}[theorem]{Corollary}

\theorembodyfont{\rmfamily}
\newtheorem{proof}{\textmd{\textit{Proof.}}}

\newtheorem{remark}[theorem]{Remark}

\newtheorem{acknowledgement}{\textmd{\textit{Acknowledgements.}}}

\makeatletter

\@addtoreset{equation}{section}
\makeatother

\newcommand{\qedd}{\hfill \Box}

\newcommand{\R}{\ensuremath{\mathbb{R}}}
\newcommand{\Sph}{\ensuremath{\mathbb{S}}}

\newcommand{\cC}{\ensuremath{\mathcal{C}}}

\title{Riemannian and Finslerian spheres with \\ fractal cut loci
\footnote{
Mathematics Subject Classification (2010)\,:\,53C60, 53C22.}
}
\author{Jinichi ITOH, Sorin V. SABAU}
\date{}
\begin{document}


\maketitle

\begin{abstract}
The present paper shows that for a given integer $k\geq 2$ it is possible to construct an at least $k$-differentiable Riemannian metric on the sphere of a certain dimension such that the cut locus of a point of it
 becomes a fractal. Moreover, we show that this construction can be extended to the case of Finsler sphere as well.
\end{abstract}

\section{Introduction}\label{sec:1}

 The cut locus $\cC(p)$ of a point $p$ in a Riemannian or Finsler manifold is, roughly speaking, 
 the set of all other points for which there are multiple minimizing geodesics connecting them from $p$. Of course, in some special cases, it may contain additional points where the minimizing geodesic is unique. 

The notion of cut locus was introduced and studied for the first time by H. Poincar\'e in 1905 for the Riemannian case. Later on,  in the case of a two dimensional analytical sphere, S. B. Myers has proved
in 1935 that the cut locus of a point is a finite tree (\cite{M}) in both Riemannian and Finslerian cases. Moreover, in the case of an analytic Riemannian manifold, M. Buchner has shown the triangulability of $\cC(p)$ (\cite{Buch1}), and has determined its local structure for the low dimensional case (\cite{Buch2}) in 1977 and 1978, respectively. 

Despite of the vast literature  existing for the Riemannian case, the investigations of the cut locus of a Finsler manifolds are scarce (see \cite{BCS}, \cite{S}, \cite{LN}).

Recently, it was shown that the cut locus of a closed subset $N$ of a Finsler surface has the structure of a local tree being a union of rectifiable Jordan arcs (\cite{TS}). Even though the results are similar to the Riemannian case, showing that there is nothing special about the metric structure to be Riemannian, one should pay always attention to the fact that, unlike its Riemannian correspondent,  the Finslerian distance is not symmetric, so the proofs and arguments are quite different.

Returning to the Riemannian case, in the case of an arbitrary metric, the cut locus of a point can have a very complicated structure. For example, H. Gluck and D. Singer have constructed a $C^\infty$ Riemannian manifold that has a point whose cut locus is not triangulable (\cite{GS}).

There is a closed relationship between the complexity of the cut locus and the regularity of the metric, regardless it is Riemannian or Finslerian.
Indeed, if the metric has a certain degree of regularity, then the cut locus of a point may enjoy a simple structure. However, if the metric loses its regularity, the cut locus might become a very complicated set, for example a fractal. Recall that, roughly speaking, a fractal is a set whose Hausdorff dimension is not an integer (see \cite{Fal} for alternative definitions and examples), fact that make fractals typical examples of what we call \lq\lq complicated sets". 

Let us mention that the cut locus of any point on a $C^\infty$ Riemannian manifold can not be a fractal (see \cite{IT}). However, there is a $C^{1,1}$ Riemannian metric on the two dimensional sphere $\Sph^2$ and a point $p\in \Sph^2$ such that the total length of $\cC(p)$ is infinite (see \cite{I1}).  

Motivated by all these, in the present paper, we are going to study the following two questions:
{\it 
\begin{enumerate}
\item There exists Riemannian metrics having points whose cut locus is a fractal?
\item There are more general metric structures, for example Finsler metrics, with the same property?
\end{enumerate}
}

The answer to both questions above is affirmative. Indeed, in the present paper we construct an at least $k$ differrentiable, $2\leq k<\infty$,  Riemannian metric on a topological sphere $\Sph^n$, provided the dimension $n$ is high enough, namely, we prove

\begin{theorem}\label{theorem A}
For any integer $2\leq k<\infty$ there is an at least $k$-differentiable Riemannian metric on the $n(k)$-dimensional sphere $\Sph^{n(k)}$ and a point $p$ in $\Sph^{n(k)}$ such that the Hausdorff dimension of $\cC(p)$ is a real number between 1 and 2, where $n(k):=\frac{3^{k+1}}{2}+1$. 
\end{theorem}

Moreover, we show that there is a Finsler metric of Randers type on this sphere with the same property. Indeed, if we use the same notations as in Theorem \ref{theorem A}, we have

 \begin{theorem}\label{theorem B}
For any integer $2\leq k<\infty$, under the influence of a suitable magnetic field $\beta$ defined on $\Sph^{n(k)}$, there is an at least $k$-differentiable non-Riemannian Finsler metric of Randers type on $\Sph^{n(k)}$ such that the cut locus of the point $p$ with respect to this Finsler metric coincides with $\cC(p)$.
\end{theorem}

\acknowledgement{ We express our gratitude to H. Shimada for many useful discussions during the preparation of this manuscript.}


\section{Basic construction}\label{sec:2}                                  

In this section, we will construct:
\begin{itemize}
\item an infinite tree $IT$ in $\R^n$ with the end points set $E$,
\item a closed, convex ball $H$ in $\R^n$ with $C^1$-boundary that contains $IT$ and $E\subset \partial H$.
\end{itemize}

The set $E$ is actually a fractal whose Hausdorff dimension is a value between 1 and 2 (see \cite{Fal} for definitions).

We begin by defining three infinite series of numbers 
\begin{equation}\label{def t, l, r}
\begin{split}
& t_i  :=\Bigl(\frac{1}{3^{k-1}}\Bigr)^i, \qquad i\in \{0,1, \dots\}\\
& l_i  :=\frac{t_i}{\sin\bigl(\dfrac{\phi}{3^i}\bigr)},\ \qquad i\in \{0,1, \dots\}\\
& r_i :=\sum_{\nu=i+1}^\infty l_\nu\cos(\frac{\phi}{3^\nu}),\qquad i\in \{-1,0,1, \dots\},
\end{split}
\end{equation}
where $\phi\in(0,\dfrac{\pi}{2})$ is an arbitrary fixed angle and $k\geq 2$ a fixed integer.

One can easily see that $(t_i)$, $(l_i)$ and $(r_i)$  are monotone decreasing series that converge to zero, for $i\to\infty$.

We will use these in order to construct a fractal set in $\R^n$. For the moment we do not assume any relation between $n$ and $k$. 

Let us consider in $\R^n$ the points $o$, $q$, $q_0$, $\dots$, $q_{j_1j_2\dots j_m}$, where 
$m\in\{1,2,\dots\}$, $i\in \{1,2,\dots,m\}$, and $j_i\in\{-(n-1),\dots,-1,0,1,\dots,n-1\}$, defined as follows.
\begin{equation}
o:=(0,0,\dots,0),\qquad q:=(l_0,0,\dots,0),
\end{equation}

$\boxed{m=1}$
\begin{equation}
\begin{split}
& q_0:=(l_0+l_1,0,\dots,0)\\
& q_{j_1}:=(l_0+l_1\cos(\frac{\phi}{3}),0,\dots,0,l_1\sin(\frac{\phi}{3})a_1,0,\dots,0),\\
& \qquad\qquad\qquad \qquad\qquad\qquad\quad \hat{\ ^{|j_1|+1}}
\end{split}
\end{equation}
\qquad where $j_1\in\{-(n-1),\dots,-1,0,1,\dots,n-1\}$, the symbol $ \hat{}$ shows the position of a component in a vector, and $|\cdot|$ is the usual absolute value of a real number; 

\bigskip
$\boxed{m>1}$
\begin{equation}
\begin{split}
& q_{\underbrace{0\dots0}_m}:=(\sum_{i=0}^ml_i,0,\dots,0)\\
& q_{\underbrace{0\dots j_m}_m}:=(\sum_{i=0}^{m-1}l_i+l_m\cos(\frac{\phi}{3^m}),0,\dots,0,l_1\sin(\frac{\phi}{3^m})a_m,0,\dots,0),\\
& \qquad\qquad\quad \quad\qquad\qquad\qquad\qquad\qquad\qquad\quad \hat{\ ^{|j_m|+1}}\\
& q_{j_1j_2\dots j_m}:=R^{|j_1|+1}(q,\frac{\phi}{3}a_1)\circ R^{|j_2|+1}(q,\frac{\phi}{3^2}a_2)\circ
\dots\\
&\qquad \qquad \quad \circ R^{|j_{m-1}|+1}(q_{j_1j_2\dots j_{m-1}}\frac{\phi}{3^{m-1}}a_{m-1})\circ q_{00\dots j_m},
\end{split}
\end{equation}
where 
\begin{equation}
a_i:=
\begin{cases}
-1,\qquad \textrm{ if}\ j_i<0\\
\ \ 0,\qquad \textrm{ if}\ j_i=0\\
\ \ 1,\qquad \textrm{ if}\ j_i> 0
\end{cases},
\end{equation}
and $R^j(x,\theta)$ is the rotation of angle $\theta$ around the affine subspace that is orthogonal to the $<e_1,e_j>$ plane and contains $x$. Here $e_j$ is the unit vector with all components zero except the  $j$-th component which is 1, namely
\begin{equation*}
\begin{split}
& e_j=(0,\dots,1,\dots,0),\\
& \ \qquad\qquad\quad \hat{^{j}}
\end{split}
\end{equation*}
$x\in \{q_{j_1j_2\dots j_m}: j_1, j_2, \dots j_m\ \textrm{ as above}\}$ and $\theta\in\{\dfrac{\phi}{3^i}:i=1,\dots,m-1\}$.

One can easily see that the points $q_{j_{1}j_{2}\cdots j_{m-1}*}$ are all on the sphere with center $q_{j_{1}j_{2}\cdots j_{m-1}}$ and radius $l_{m}$, for any fixed $m>0$ and $j_{i}$, $i\in\{1,2,\dots,m\}$. 

For later use we denote the segment between $o$ and $q$ by $s$, the segment between $q$ and $q_{j_{1}}$ by $s_{j_{1}}$ and so on inductively, such that the segment between $q_{j_{1}j_{2}\cdots j_{m-1}}$ and $q_{j_{1}j_{2}\cdots j_{m}}$ will be denoted by $s_{j_{1}j_{2}\cdots j_{m-1}j_{m}}$. 

Likely, we denote the ray from $o$ that contains $q$ by $\gamma$, the ray from $q$ that contains $s_{j_{1}}$ by $\gamma_{j_{1}}$ and so on inductively, such that the ray from $q_{j_{1}j_{2}\cdots j_{m-1}}$ that contains $s_{j_{1}j_{2}\cdots j_{m-1}j_{m}}$ will be denoted by $\gamma_{j_{1}j_{2}\cdots j_{m-1}j_{m}}$. 

Let $IT_{0}:=\{s\}$, $IT_{m}:=\cup_{i=1}^{m}\{s_{j_{1}j_{2}\cdots j_{m-1}j_{m}}\}$ be the union of segments $s_{j_{1}j_{2}\cdots j_{m-1}j_{m}}$, $IT:=\lim_{m\to \infty}IT_{m}$ be the infinite tree and let $Q:=\cup_{m=1}^{\infty}\{q_{j_{1}j_{2}\cdots j_{m-1}j_{m}}\}$ be the set of points $q_{j_{1}j_{2}\cdots j_{m-1}j_{m}}$.

Taking into account \eqref{def t, l, r} and the construction above, one can see that $\lim_{i\to \infty}l_{1}=0$ implies $L:=\sum_{i=0}^{\infty}l_{i}$ is finite, therefore the edges can not prolong to infinity, so the set $Q$ must have a subset of limit points $E\subset Q$. 

The set $E$ is in fact the set of {\it end points} of the infinite tree $IT$, except the root point $o$ (see Figure 1). Moreover, one can see that $IT$ is completely contained in the ball with center $o$ and radius $L$ .

\bigskip


\begin{tikzpicture}
  [grow cyclic,
   level 1/.style={level distance=12mm,sibling angle=60}]
  \coordinate [rotate=-90] 
    child foreach \x in {1,2,3};
      \draw (-1.5,0) -- (0,0);
      \pgftext[base,x=-1.5cm,y=-0.5cm] {$o$};
      \pgftext[base,x=0cm,y=-0.5cm] {$q$};
       \pgftext[base,x=0.3cm,y=1cm] {$q_1$};
        \pgftext[base,x=1cm,y=-0.5cm] {$q_0$};
        \pgftext[base,x=0.2cm,y=-1.2cm] {$q_{-1}$};
        \pgftext[base,x=0cm,y=-2cm] {$\ $};
\end{tikzpicture}
\qquad\qquad
\begin{tikzpicture}
  [grow cyclic,
   level 1/.style={level distance=12mm,sibling angle=60},
   level 2/.style={level distance=6mm,sibling angle=40}]
  \coordinate [rotate=-90] 
    child foreach \x in {1,2,3}
      {child foreach \x in {1,2,3}};
      \draw (-1.5,0) -- (0,0);
      \pgftext[base,x=-1.5cm,y=-0.5cm] {$o$};
      \pgftext[base,x=0cm,y=-0.5cm] {$q$};
       \pgftext[base,x=0.3cm,y=1cm] {$q_1$};
       \pgftext[base,x=1cm,y=-0.5cm] {$q_0$};
       \pgftext[base,x=0.2cm,y=-1.2cm] {$q_{-1}$};
        \pgftext[base,x=0.2cm,y=1.7cm] {$q_{11}$};
         \pgftext[base,x=1cm,y=1.7cm] {$q_{10}$};
         \pgftext[base,x=1.7cm,y=1.3cm] {$q_{1-1}$}; 
         \pgftext[base,x=0.2cm,y=-1.8cm] {$q_{-11}$};
         \pgftext[base,x=1.2cm,y=-1.8cm] {$q_{-10}$};
         \pgftext[base,x=1.8cm,y=-1.3cm] {$q_{-1-1}$}; 
         \pgftext[base,x=2cm,y=0.4cm] {$q_{01}$};
         \pgftext[base,x=2.2cm,y=-0.2cm] {$q_{00}$};
         \pgftext[base,x=2.2cm,y=-0.6cm] {$q_{0-1}$}; 
          \pgftext[base,x=0cm,y=-2cm] {$\ $};
\end{tikzpicture}
\qquad\qquad
\begin{tikzpicture}
  [grow cyclic,
   level 1/.style={level distance=12mm,sibling angle=60},
   level 2/.style={level distance=6mm,sibling angle=40},
   level 3/.style={level distance=4mm,sibling angle=25}]
  \coordinate [rotate=-90] 
    child foreach \x in {1,2,3}
      {child foreach \x in {1,2,3}
        {child foreach \x in {1,2,3}}};
         \draw (-1.5,0) -- (0,0);
         \pgftext[base,x=-1.5cm,y=-0.5cm] {$o$};
      \pgftext[base,x=0cm,y=-0.5cm] {$q$};
       \pgftext[base,x=0.3cm,y=1cm] {$q_1$};
       \pgftext[base,x=1cm,y=-0.5cm] {$q_0$};
       \pgftext[base,x=0.2cm,y=-1.2cm] {$q_{-1}$};
\end{tikzpicture}

\bigskip

\bigskip

\begin{tikzpicture}
  [grow cyclic,
   level 1/.style={level distance=12mm,sibling angle=60},
   level 2/.style={level distance=6mm,sibling angle=40},
   level 3/.style={level distance=4mm,sibling angle=25},
   level 4/.style={level distance=2mm,sibling angle=20} ]
  \coordinate [rotate=-90] 
    child foreach \x in {1,2,3}
      {child foreach \x in {1,2,3}
        {child foreach \x in {1,2,3}
        {child foreach \x in {1,2,3}}}};
         \draw (-1.5,0) -- (0,0);
         \pgftext[base,x=-1.5cm,y=-0.5cm] {$o$};
      \pgftext[base,x=0cm,y=-0.5cm] {$q$};
       \pgftext[base,x=0.3cm,y=1cm] {$q_1$};
       \pgftext[base,x=1cm,y=-0.5cm] {$q_0$};
       \pgftext[base,x=0.2cm,y=-1.2cm] {$q_{-1}$};
        \pgftext[base,x=0cm,y=-2.4cm] {$\ $};
\end{tikzpicture}
\qquad\qquad
\begin{tikzpicture}
  [grow cyclic,
   level 1/.style={level distance=12mm,sibling angle=65},
   level 2/.style={level distance=6mm,sibling angle=40},
   level 3/.style={level distance=4mm,sibling angle=25},
   level 4/.style={level distance=2mm,sibling angle=20},
   level 5/.style={level distance=1mm,sibling angle=10} ]
  \coordinate [rotate=-90] 
    child foreach \x in {1,2,3}
      {child foreach \x in {1,2,3}
        {child foreach \x in {1,2,3}
        {child foreach \x in {1,2,3}
        {child foreach \x in {1,2,3}}}}};
         \draw (-1.5,0) -- (0,0);
         \pgftext[base,x=-1.5cm,y=-0.5cm] {$o$};
      \pgftext[base,x=0cm,y=-0.5cm] {$q$};
       \pgftext[base,x=0.3cm,y=1cm] {$q_1$};
       \pgftext[base,x=1cm,y=-0.5cm] {$q_0$};
       \pgftext[base,x=0.2cm,y=-1.2cm] {$q_{-1}$};
\end{tikzpicture}

\bigskip

{\bf Figure 1.} The tree $IT_{m}$ in $\R^{2}$ for $m=1,2,3,4,5$, respectively.


\begin{remark}
Remark that for the infinite tree $IT$ in $\R^{n}$, there are $2n-1$ branches that ramify from each node. The maximum depth level in the tree is given by $m$ and the branches and nodes at a given depth level $i$ are specified by $j_{i}$. Figure 1 shows the growth of the tree IT in the case $\R^{2}$, namely a tree with three branches that ramify from each node at each level. 
\end{remark}

Next, we will construct a closed, convex ball $H$ in $\R^{n}$ with $C^{1}$-boundary as follows:
\begin{itemize}
\item take a point $\hat q$ on the straight line $\gamma$ such that $d(0,\hat q)=\frac{r_{-1}}{\cos \phi}$, where $d$ is the usual Euclidean distance;
\item consider the right circular cone $C$ with vertex $\hat q$, axis $\gamma$ and vertex angle $\pi-2\phi$; 
\item consider the $(n-1)$-spheres $S_{0}=S(q,r_{0})$ and $S=S(o,r_{-1})$ in $\R^{n}$ of center $q$ and $o$, and radii $r_{0}$ and $r_{-1}$, respectively.
\end{itemize}

Then, it can be verified by simple trigonometric computations that the right circular cone $C$ is exterior tangent to the spheres $S_{0}$ and $S$ (see Figure 2). The intersection of $C$ with the spheres $S_{0}$ and $S$ is made of the $(n-2)$-spheres $c$ and $c'$, respectively (in the case $IT\subset \R^{3}$ these are circles).  

\bigskip


\begin{tikzpicture}[scale=0.6]
\draw (0,0) circle (7cm);
\draw (4,0) circle (4.5cm);
\draw (-7,0) -- (13,0);
\filldraw (0,0) circle (2pt);
\filldraw (4,0) circle (2pt);
\draw (11.25,0) -- (0,9);
\draw (11.25,0) -- (0,-9);
\draw (0,0) -- (4.45,5.45);
\draw (6.95,3.45) -- (4,0);
\draw (0,0) -- (4.45,-5.45);
\draw (6.95,-3.45) -- (4,0);
\draw (0,-0.35) node [anchor=east] {$o$};
\draw (4,-0.35) node [anchor=east] {$q$};
\draw (11.7,-0.35) node [anchor=east] {$\hat q$};
\draw [ultra thick] (4,0) -- (5.5,0);
\filldraw (5.5,0) circle (2pt);
\draw (5.85,0.2) node [anchor=north] {$q_{0}$};
\draw [ultra thick] (4,0) -- (5.3,0.5);
\filldraw (5.3,0.5) circle (2pt);
\draw (5.3,-0.5) node [anchor=north] {$q_{-1}$};
\draw [ultra thick] (4,0) -- (5.3,-0.5);
\filldraw (5.3,-0.5) circle (2pt);
\draw (5.7,1) node [anchor=north] {$q_{1}$};
\draw [ultra thick] (0,0) -- (4,0.0);
\draw (2,0) node [anchor=south] {$l_{0}$};
\draw (2.65,2) node [anchor=south] {$r_{-1}$};
\draw (4.7,1) node [anchor=south] {$r_{0}$};
\draw (-1,7) node [anchor=south] {$S$};
\draw (1,3.5) node [anchor=south] {$S_{0}$};
\draw (12,0) node [anchor=south] {$\gamma,\gamma_{0}$};
\draw (6,4.2) node [anchor=south] {$t_{0}$};
\draw (10,1.5) node [anchor=south] {$r_{0}\tan \phi$};
\draw (11,0.2) arc (90:180:0.2);
\draw (9.8,0) node [anchor=south] {$\frac{\pi}{2}-\phi$};
\draw (4.6,5.4) node [anchor=south] {$\hat c'$};
\draw (7,3.6) node [anchor=south] {$\hat c$};
\end{tikzpicture}
\bigskip

{\bf Figure 2.} A longitudinal section in the cone $C$ in the case $IT_{1}\subset \R^{3}$. 

\bigskip

Remark that the $(n-2)$-spheres $c$ and $c'$ cut out:
\begin{itemize}
\item a truncated cone $A\subset C$, from the cone $C$;
\item a small spherical cap $P_{0}$ with boundary $c\cap S_{0}$ from the sphere $S_{0}$;
\item a large spherical cap $P$ with boundary $c'\cap S$ from the sphere $S$.
\end{itemize}

Gluing at the both ends of the truncated cone $A_{0}$ the spherical caps mentioned above we obtain a closed, convex ball $H_{0}\subset \R^{n}$ whose boundary is a $C^{1}$-hypersurface $\partial H_{0}\subset \R^{n}$ (see Figure 3).


\bigskip

\begin{tikzpicture}[scale=0.6]
\draw (0,0) circle (7cm);
\draw (4,0) circle (4.5cm);
\filldraw (0,0) circle (2pt);
\filldraw (4,0) circle (2pt);
\draw (0,0) -- (4.45,5.45);
\draw (4.45,5.45) -- (6.95,3.45);
\draw (4.45,-5.45) -- (6.95,-3.45);
\draw (6.95,3.45) -- (4,0);
\draw (0,0) -- (4.45,-5.45);
\draw (6.95,-3.45) -- (4,0);
\draw (0,-0.35) node [anchor=east] {$o$};
\draw (4.5,-0.35) node [anchor=east] {$q$};
\draw [ultra thick] (0,0) -- (4,0.0);
\draw (-1,7) node [anchor=south] {$S$};

\draw (1.9,4) node [anchor=south] {$S_{0}$};

\draw (4.6,5.4) node [anchor=south] {$c'$};
\draw (7,3.6) node [anchor=south] {$c$};
\draw [thick,densely dashed] (4.5,0) ellipse (1 and 5.3);
\draw [thick, densely dashed] (7,0) ellipse (0.5 and 3.35);
\draw [gray,ultra thin] (1,0) -- (5.1,4.9);
\draw [gray,ultra thin] (1,0) -- (5.1,-4.9);
\draw [gray,ultra thin] (2,0) -- (5.85,4.35);
\draw [gray,ultra thin] (2,0) -- (5.85,-4.35);
\draw [gray,ultra thin] (3,0) -- (6.4,3.9);
\draw [gray,ultra thin] (3,0) -- (6.4,-3.9);
\draw (1.4,-0.35) node [anchor=east] {$\hat u$};
\draw [thick, densely dotted] (0,0) -- (0,7);
\draw [thick, densely dotted] (0,0) -- (0,-7);
\draw [thick, densely dotted] (0,0) -- (-7,0);
\draw [thick, densely dotted] (0,0) -- (-4.9,4.9);
\draw [thick, densely dotted] (0,0) -- (-4.9,-4.9);
\draw [thick, dotted] (4,0) -- (7.7,2.7);
\draw [thick, dotted] (4,0) -- (7.7,-2.7);
\draw [thick, dotted] (4,0) -- (8.2,1.5);
\draw [thick, dotted] (4,0) -- (8.2,-1.5);
\draw [thick, dotted] (4,0) -- (8.5,0);
\draw [thick,densely dashed] (0,0) ellipse (1.3 and 7);
\draw (6,4.2) node [anchor=south] {$A$};
\end{tikzpicture}
\bigskip

{\bf Figure 3.} The segment $s=oq$ is the inward cut locus of the $C^{1}$-hypersurface $\partial H_{0}$ in the case $IT_{0}\subset \R^{3}$.

\bigskip


We remark that the (inward) cut locus of $H_{0}$ endowed with the induced Euclidean norm from $\R^{n}$ is exactly the segment $s$. Indeed, the inward geodesic rays from $\partial H_{0}$ concentrates at a point $\hat u\in s$. 
The same length inward geodesic rays orthogonal to the cloth of the truncated cone $A_{0}$, emanating from an arbitrary point of the $(n-2)$-sphere 
$$c_{u}:=\{x\in A|d(x,\hat q)=(d(o,\hat q)-u)\sin\phi\},$$
 concentrate at a point $\hat u:=(u,0,\dots,0)\in s\setminus \{o,q\}$, the inward geodesic rays of same length from the small spherical cap of $S_{0}$ concentrate at $q$ and similarly the geodesic rays from the large spherical cap of $S$ concentrate at $o$.

Therefore, given a segment $s=oq\subset IT$ we obtain a 
$C^{1}$-hypersurface $\partial H_{0}\in \R^{n}$ whose (inward) cut locus and conjugate locus
is exactly the segment $s$.

\bigskip

Let us see how this construction looks like at next step. Consider $m=1$ and therefore our tree becomes $IT_{1}=\{s,s_{j_{1}}|j_{1}\in\{-(n-1),\dots,0,1,\dots, n-1\}\}$. The construction reads now:
\begin{itemize}
\item take a set of points $\hat q_{j_{1}}$ on the straight line $\gamma_{j_{1}}$ such that 
$d(q,\hat q_{j_{1}})=\frac{r_{0}}{\cos \frac{\phi}{ 3}}$;
\item consider the right circular cones $C_{j_{1}}$ with vertices $\hat q_{j_{1}}$, axes 
$\gamma_{j_{1}}$ and vertex angles $\pi-2\frac{\phi}{ 3}$; 
\item consider the $(n-1)$-spheres $S_{j_{1}}=S(q_{j_{1}},r_{1})$ and $S_{0}=S(q,r_{0})$ in $\R^{n}$.
\end{itemize}

The right circular cones $C_{j_{1}}$ are exterior tangent to the spheres $S_{j_{1}}$ and $S_{0}$. The intersection of $C_{j_{1}}$ with the spheres $S_{j_{1}}$ and $S_{0}$ is made of $2n-1$ spheres $c_{j_{1}}:=C_{j_{1}}\cap S_{j_{1}}$ and $c^{0}_{j_{1}}=:C_{j_{1}}\cap S_0$, respectively, that cut off
\begin{itemize}
\item $2n-1$ truncated cones $A_{j_{1}}\subset C_{j_{1}}$, from the cone $C_{j_{1}}$;
\item $2n-1$ small spherical caps, from the spheres $S_{j_{1}}$, whose boundaries are $C_{j_{1}}\cap S_{j_{1}}$;
\item $2n-1$ large spherical caps, from the sphere $S_{0}$, whose boundaries are $C_{j_{1}}\cap S_{0}$.
\end{itemize}

Consider now the small spherical cap on $S_{0}$ cuted off by the cone $C$ on which the $2n-1$ truncated cones $A_{j_{1}}$ sit. We denote by $B_{0}\subset S_{0}$ the region left from this small spherical cap after cutting off the new appeared small spherical cups $\bigcup_{j_{1}}(C_{j_{1}}\cap S_{0})$.

Then
\begin{equation}
\partial  H_{1}:=\bigcup_{j_{1}}(A_{j_{1}}\cup B_{j_{1}})\cup P
\end{equation}
is the convex $C^{1}$-hypersurface  and define $H_{1}$ to be the closed, convex ball whose boundary is $\partial H_{1}$. Obviously $H_{1}$ contains $IT_{1}$ and the inner cut locus of $\partial H_{1}$ is precisely
$IT_{1}$.
 


This construction can be repeated for each segment $s_{j_{1}j_{2}\cdots j_{m}}\subset IT$ obtaining in this way a closed, convex ball  $H_{m}\subset \R^{n}$, which contains $IT_{m}$, with  $C^{1}$-boundary $\partial H_{m}$ whose (inward) cut locus coincides with the  tree $IT_{m}$. 

Indeed, we construct as follows:
\begin{itemize}
\item take a point $\hat q_{j_{1}j_{2}\cdots j_{m}}$ on the straightline $\gamma_{j_{1}j_{2}\cdots j_{m}}$ such that 
\begin{equation}
d(q_{j_{1}j_{2}\cdots j_{m-1}},\hat q_{j_{1}j_{2}\cdots j_{m}})=\frac{r_{m-1}}{\cos (\frac{\phi}{3^{m}})};
\end{equation}
\item consider the right circular cone $C_{j_{1}j_{2}\cdots j_{m}}$ with vertex $\hat q_{j_{1}j_{2}\cdots j_{m}}$, axis $\gamma_{j_{1}j_{2}\cdots j_{m}}$ and vertex angle $\pi-2\frac{\phi}{3^{m}}$;
\item denote the $(n-1)$-sphere $S_{j_{1}j_{2}\cdots j_{m}}:=S(q_{j_{1}j_{2}\cdots j_{m}},r_{m})$ and consider the spheres $S_{j_{1}j_{2}\cdots j_{m}}$ and $S_{j_{1}j_{2}\cdots j_{m-1}}$ with centers at the ends of the segment $s_{j_{1}j_{2}\cdots j_{m}}$. 
\end{itemize}

It follows again that the right circular cone $C_{j_{1}j_{2}\cdots j_{m}}$ is exterior tangent to the spheres $S_{j_{1}j_{2}\cdots j_{m-1}}$ and $S_{j_{1}j_{2}\cdots j_{m}}$. The intersection of $C_{j_{1}j_{2}\cdots j_{m}}$ with the spheres $S_{j_{1}j_{2}\cdots j_{m-1}}$ and $S_{j_{1}j_{2}\cdots j_{m}}$ is made of the $(n-2)$-spheres $c_{j_{1}j_{2}\cdots j_{m-1}}$ and $c_{j_{1}j_{2}\cdots j_{m}}$, respectively.

Similarly as above, the $(n-2)$-spheres $c_{j_{1}j_{2}\cdots j_{m-1}}$ and $c_{j_{1}j_{2}\cdots j_{m}}$ cut out:
\begin{itemize}
\item a truncated cone $A_{j_{1}j_{2}\cdots j_{m}}\subset C_{j_{1}j_{2}\cdots j_{m}}$;
\item a small spherical cap of $S_{j_{1}j_{2}\cdots j_{m-1}}$ whose boundary is $c_{j_{1}j_{2}\cdots j_{m-1}}\cap S_{j_{1}j_{2}\cdots j_{m-1}}$;
\item a large spherical cap $S_{j_{1}j_{2}\cdots j_{m}}$ whose boundary is $c_{j_{1}j_{2}\cdots j_{m}}\cap S_{j_{1}j_{2}\cdots j_{m}}$.
\end{itemize}

Consider now the small spherical cap on $S_{j_{1}j_{2}\cdots j_{m-1}}$ cuted off by the cone $C_{j_{1}j_{2}\cdots j_{m}}$ on which the $2n-1$ truncated cones $A_{j_{1}j_{2}\cdots j_{m}}$ sit. We denote by $B_{j_{1}j_{2}\cdots j_{m}}\subset S_{j_{1}j_{2}\cdots j_{m-1}}$ the region left from this small spherical cap after cutting off the new appeared small spherical cups
 $\cup_{j_{1}j_{2}\cdots j_{m}}(C_{j_{1}j_{2}\cdots j_{m}}\cap S_{j_{1}j_{2}\cdots j_{m}})$.

Then
\begin{equation}
\partial H_{m}:=\cup_{i=0}^{m}\cup_{j_{1}j_{2}\cdots j_{m}}(A_{j_{1}j_{2}\cdots j_{m}}\cup B_{j_{1}j_{2}\cdots j_{m}})\cup P
\end{equation}
is the convex $C^{1}$-hypersurface that defines the closed, convex ball $H_{m}$ in $\R^{n}$. By a similar argument as above one can see that 
$\cup_{i=0}^{m}\cup_{j_{1}j_{2}\cdots j_{m}}s_{j_{1}j_{2}\cdots j_{m}}$ is the inward cut locus of $\partial H_{m}$. One can remark that in our construction the conjugate locus coincides with the focal locus.

 Indeed, the equal length inward geodesic rays orthogonal to the cloth of the truncated cone $A_{j_{1}j_{2}\cdots j_{m}}$, emanating from an arbitrary point of the $(n-2)$-sphere 
$\{x\in A_{j_{1}j_{2}\cdots j_{m}}|d(x,\hat q_{j_{1}j_{2}\cdots j_{m}})=constant\}$,
 concentrate at a point on the open segment $s_{j_{1}j_{2}\cdots j_{m}}$.  The inward geodesic rays of same length orthogonal to the region $B_{j_{1}j_{2}\cdots j_{m}}\subset S_{j_{1}j_{2}\cdots j_{m}}$
 concentrate at $q_{j_{1}j_{2}\cdots j_{m}}$ and similarly the geodesic rays from the large spherical cap of $S$ concentrate at $o$.
 
Finally, we define
\begin{equation}
H:=\lim_{m\to\infty}H_{m}
\end{equation}
that have the required properties. 

Therefore we obtain

\begin{proposition}
Let $IT$ be the infinite tree in $\R^{n}$ and $H$ the closed, convex ball with $C^{1}$-boundary  constructed above. Then  $H\subset \R^{n}$ contains $IT$ and  $E\subset \partial H$ by construction. 
\end{proposition}

\section{The Haussdorf measure of $E$}\label{sec:3}

In this section we will investigate the Hausdorff dimension of the set $E$. The idea is to construct a set of points $\hat E$ whose Hausdorff dimension $\dim_{\mathcal H}\hat E$ can be easily computed  and such that $\dim_{\mathcal H} E=\dim_{\mathcal H}\hat E$.

We start by defining an infinite series of numbers $(\alpha_{i})$, $i\in\{0,1, \dots\}$ 
by
\begin{equation}
\alpha_{0}=\frac{3^{k-2}}{3^{k-2}-1},\quad \alpha_{i+1}=\frac{1}{3^{k-1}}\alpha_{i},
\end{equation}
where $t_{i}$ is defined in \eqref{def t, l, r}.

Remark that, for any $i$, we have
\begin{equation}
\alpha_{i}=3\alpha_{i+1}+t_{i}.
\end{equation}

In $\R^{n-1}$ we consider 
\begin{itemize}
\item the concentric $(n-2)$-balls $\hat C=B(o,\alpha_{0})$ and $\hat S=B(o,\alpha_{0}-t_{0})$ with the center $o$ at the origin of $\R^{n-1}$ and radii $\alpha_{0}$ and $\alpha_{0}-t_{0}$, respectively; 
\item $\hat A:=\hat C\setminus \hat S$ the annulus obtained by removing $\hat S$ from $\hat C$.
\end{itemize}

We will construct a set of points $y_{j_{1}j_{2}\cdots j_{m}}\in \hat S$, where $m\in \{1,2,\dots\}$ and $j_{m}\in \{-(n-1),...,-1,0,1,...,n-1\}$ are as in Section 2. 

For $m=1$ we consider the points 
\begin{equation}
(y_{j_{1}})=(y_{-(n-1)},y_{-(n-2)},\dots,y_{-2},y_{-1},y_{0},y_{1},y_{2},\dots,y_{n-1})\in \R^{n-1}, 
\end{equation}
where 
\begin{equation}\label{y_{j_{1}}}
\begin{split}
&(y_{1},y_{2},\dots,y_{n-1})^{t}=2\alpha_{1}I_{n-1},\quad y_{0}=o,\\
& (y_{-(n-1)},y_{-(n-2)},\dots,y_{-2},y_{-1})^{t}=-2\alpha_{1}I_{n-1},
\end{split}
\end{equation}
where $^{t}$ represents the transposed of a vector, and $I_{n-1}$ the identity matrix.

We can construct iteratively, for arbitrary $m$, the set of points $y_{j_{1}j_{2}\cdots j_{m}}\in \R^{n-1}$ defined by
\begin{equation}
y_{j_{1}j_{2}\cdots j_{m}}=(\sum_{i=1}^{m}2b_{i}^{1}\alpha_{i},\sum_{i=1}^{m}2b_{i}^{2}\alpha_{i},\dots, \sum_{i=1}^{m}2b_{i}^{n-1}\alpha_{i}),
\end{equation}
where $(b_{i}^{j})$ is an $m\times (n-1)$-real matrix defined by
\begin{equation}
b_{i}^{j}=
\begin{cases}
& \ \ 1,\quad \textrm{if} \ j_{i}=j\\
& -1 ,\quad \textrm{if} \ j_{i}=-j\\
& \ \ 0,\quad \textrm{otherwise} ,
\end{cases}
\end{equation}
for all $i\in\{1,2,\dots,m\}$ and $j\in \{1,2,\dots,n-1\}$. One can easily see that for $m=1$ we get \eqref{y_{j_{1}}}.


\bigskip
\begin{tikzpicture}[scale=0.4]
\draw (0,0) circle (5.5cm);
\draw (0,0) circle (4.5cm);
\filldraw (0,0) circle (1pt);
\draw (0,0) node [anchor=north] {$o$};
\draw (-4.5,4.5) node [anchor=north] {$c'$};
\draw (-3.2,3.2) node [anchor=north] {$c$};
\draw (2,2) node [anchor=north] {$P_{0}$};
\end{tikzpicture}
%
\begin{tikzpicture}[scale=0.4]
\draw (0,0) circle (5.5cm);
\draw (0,0) circle (4.5cm);
\filldraw (0,0) circle (1pt);
\draw (0,0) node [anchor=north] {$o$};
\draw (0,0) circle (1.5cm);
\draw (3,0) circle (1.5cm);
\draw (-3,0) circle (1.5cm);
\draw (0,3) circle (1.5cm);
\draw (0,-3) circle (1.5cm);
\draw (0,0) circle (1cm);
\draw (3,0) circle (1cm);
\draw (-3,0) circle (1cm);
\draw (0,-3) circle (1cm);
\draw (0,3) circle (1cm);
\filldraw (0,3) circle (1pt);
\filldraw (3,0) circle (1pt);
\filldraw (-3,0) circle (1pt);
\filldraw (0,-3) circle (1pt);
\draw (2.5,2.7) node [anchor=north] {$B_{0}$};
\draw (4,3.8) node [anchor=north] {$A_{0}$};
\end{tikzpicture}
%
\begin{tikzpicture}[scale=0.4]
\draw (0,0) circle (5.5cm);
\draw (0,0) circle (4.5cm);
\filldraw (0,0) circle (1pt);
\draw (0,0) circle (1.5cm);
\draw (3,0) circle (1.5cm);
\draw (-3,0) circle (1.5cm);
\draw (0,3) circle (1.5cm);
\draw (0,-3) circle (1.5cm);
\draw (0,0) circle (1cm);
\draw (3,0) circle (1cm);
\draw (-3,0) circle (1cm);
\draw (0,-3) circle (1cm);
\draw (0,3) circle (1cm);
\filldraw (0,3) circle (1pt);
\filldraw (3,0) circle (1pt);
\filldraw (-3,0) circle (1pt);
\filldraw (0,-3) circle (1pt);

\filldraw (0.66,3) circle (1pt);
\filldraw (0,3.66) circle (1pt);
\filldraw (0,2.33) circle (1pt);
\filldraw (-0.66,3) circle (1pt);
\draw (0,3) circle (0.33cm);
\draw (0,3.66) circle (0.33cm);
\draw (0,2.33) circle (0.33cm);
\draw (-0.66,3) circle (0.33cm);
\draw (0.66,3) circle (0.33cm);

\filldraw (0,0.66) circle (1pt);
\filldraw (-0.66,0) circle (1pt);
\filldraw (0,-0.66) circle (1pt);
\filldraw (0.66,0) circle (1pt);
\draw (0,0) circle (0.33cm);
\draw (0,0.66) circle (0.33cm);
\draw (0,-0.66) circle (0.33cm);
\draw (-0.66,0) circle (0.33cm);
\draw (0.66,0) circle (0.33cm);

\filldraw (-3,0.66) circle (1pt);
\filldraw (-3.66,0) circle (1pt);
\filldraw (-3,-0.66) circle (1pt);
\filldraw (-2.33,0) circle (1pt);
\draw (-3,0) circle (0.33cm);
\draw (-3,0.66) circle (0.33cm);
\draw (-3,-0.66) circle (0.33cm);
\draw (-3.66,0) circle (0.33cm);
\draw (-2.33,0) circle (0.33cm);

\filldraw (3,0.66) circle (1pt);
\filldraw (2.33,0) circle (1pt);
\filldraw (3,-0.66) circle (1pt);
\filldraw (3.66,0) circle (1pt);
\draw (3,0) circle (0.33cm);
\draw (3,0.66) circle (0.33cm);
\draw (3,-0.66) circle (0.33cm);
\draw (2.33,0) circle (0.33cm);
\draw (3.66,0) circle (0.33cm);

\filldraw (0,-2.33) circle (1pt);
\filldraw (-0.66,-3) circle (1pt);
\filldraw (0,-3.66) circle (1pt);
\filldraw (0.66,-3) circle (1pt);
\draw (0,-3) circle (0.33cm);
\draw (0,-2.33) circle (0.33cm);
\draw (0,-3.66) circle (0.33cm);
\draw (-0.66,-3) circle (0.33cm);
\draw (0.66,-3) circle (0.33cm);

\end{tikzpicture}
\bigskip

{\bf Figure 4.} Mandalas seen from a far point on the ray $\gamma$ in the cases $IT_{0}=s$ (left), 
$IT_{1}=\{s,s_{j_{1}}|j_{1}\in\{-2,-1,0,1,2\}\}$ (middle) and \\
$IT_{2}=\{s,s_{j_{1}},s_{j_{1}j_{2}}|j_{i}\in\{-2,-1,0,1,2\}, i=1,2\}$ (right).

\bigskip


For each $m$ we consider
\begin{itemize}
\item the concentric $(n-2)$-balls $\hat C_{j_{1}j_{2}\cdots j_{m}}=B(o,\alpha_{m})$ and $\hat S_{j_{1}j_{2}\cdots j_{m}}=B(o,\alpha_{m}-t_{m})$;
\item $\hat A_{j_{1}j_{2}\cdots j_{m}}:=\hat C_{j_{1}j_{2}\cdots j_{m}}\setminus \hat S_{j_{1}j_{2}\cdots j_{m}}$ the annulus obtained by removing $\hat S_{j_{1}j_{2}\cdots j_{m}}$ from $\hat C_{j_{1}j_{2}\cdots j_{m}}$  (see Figure 4, 5 for the case $n=3$).
\end{itemize}

\bigskip
\begin{tikzpicture}[scale=1.3]
\draw (0,0) circle (5.5cm);
\draw (0,0) circle (4.5cm);
\draw (0,0) circle (1.5cm);
\draw (3,0) circle (1.5cm);
\draw (-3,0) circle (1.5cm);
\draw (0,3) circle (1.5cm);
\draw (0,-3) circle (1.5cm);
\draw (0,0) circle (1cm);
\draw (3,0) circle (1cm);
\draw (0,3) circle (0.33cm);
\draw (0,3.66) circle (0.33cm);
\draw (0,2.33) circle (0.33cm);
\draw (-0.66,3) circle (0.33cm);
\draw (0.66,3) circle (0.33cm);
\filldraw (0,3.66) circle (1pt);
\filldraw (-0.66,3) circle (1pt);
\filldraw (0,2.33) circle (1pt);
\filldraw (0.66,3) circle (1pt);
\draw (0,3) circle (1cm);
\draw (0,-3) circle (1cm);
\filldraw (0,0) circle (1pt);
\filldraw (0,3) circle (1pt);
\filldraw (-3,0) circle (1pt);
\filldraw (0,-3) circle (1pt);
\filldraw (3,0) circle (1pt);
\draw (0,0) node [anchor=north] {$o$};
\draw (3,0) node [anchor=north] {$y_{1}$};
\draw (-3,0) node [anchor=north] {$y_{-1}$};
\draw (0,-3) node [anchor=north] {$y_{-2}$};
\draw (0,3.1) node [anchor=north] {$y_{2}$};
\draw (0.7,3.1) node [anchor=north] {$y_{21}$};
\draw (0.4,3.8) node [anchor=north] {$y_{22}$};
\draw (2.2,-2) node [anchor=north] {$\hat S$};
\draw (0.3,1) node [anchor=north] {$\hat S_{0}$};
\draw (3,1) node [anchor=north] {$\hat S_{1}$};
\draw (0.3,1) node [anchor=north] {$\hat S_{0}$};
\draw (0.6,-2.3) node [anchor=north] {$\hat S_{-2}$};
\draw (0.8,-0.5) node [anchor=north] {$\hat A_{0}$};
\draw (0.9,-3.5) node [anchor=north] {$\hat A_{-2}$};
\draw (3.8,-0.5) node [anchor=north] {$\hat A_{1}$};
\draw (-3,-0.5) node [anchor=north] {$\hat C_{-1}$};
\draw (2.2,2.4) node [anchor=north] {$B$};
\draw (0.6,2.7) node [anchor=north] {$B_{2}$};
\draw (1,2.5) node [anchor=north] {$A_{2}$};
\draw (4.2,-2) node [anchor=north] {$\hat A$};
\draw (3.19,3.19)--(3.9,3.9);
\draw (3.7,3.6) node [anchor=north] {$t_{0}$};
\end{tikzpicture}
\bigskip

{\bf Figure 5.} A combined mandala seen from a far point on the ray $\gamma$.

\bigskip


Let us denote the set of all points $y$ by $Y:=\cup_{m=1}^{\infty}{y_{j_{1}j_{2}\cdots j_{m}}}$, and let $\hat E$ be the set of accumulation points of $Y$. Since $Y\subset \hat S$ the set of points $\hat E\subset \hat S$ exists and it is not empty.

Let $\hat R_{1}:\R^{n-1}\to\R^{n-1}$ be the contraction function that maps $\hat S$ into $\hat S_{1}$. It can be seen that $\hat R_{1}$ is a {\it similarity map}, i.e. there exists a constant $c_{1}\in (0,1)$ such that
\begin{equation}
d(\hat R_{1}(x),\hat R_{1}(y))=c_{1}d(x,y), 
\end{equation}
for any $x,y\in \R^{n-1}$. The constant $c_{1}$ is called the {\it ratio} of  $\hat R_{1}$ (see \cite{Fal}, p. 128). Indeed, taking into account our definitions, one obtains $c_{1}=\dfrac{1}{3^{k-1}}$.

Iteratively, for a given $m$, we define the mapping
\begin{equation}\label{R_{j}}
\hat R_{j}:\R^{n-1}\to\R^{n-1},\quad \hat S_{j_{1}j_{2}\cdots j_{m}}\mapsto \
\hat S_{j_{1}j_{2}\cdots j_{m}}j,
\end{equation}
for all $j\in \{-(n-1),...,-1,0,1,...,n-1\}$. Obviously $\hat R_{j}$ is a similarity map with ratio $c_{j}=\dfrac{1}{3^{k-1}}$. Thus each $\hat R_{j}$ transform subsets of $\R^{n-1}$ into geometrically similar sets.

The attractor of such a collection of similarity maps values domains, in our case $\hat E$, is a {\it self-similar set}, being a union of smaller similar copies of itself.

We recall from \cite{Fal} that the mappings $\hat R_{j}$ satisfy the {\it open set condition} if there exists a non-empty bounded open set $V$ such that
\begin{equation}
V\supset \bigcup_{j=1}^{2n-1}\hat R_{j}(V).
\end{equation}

By taking $V:=\hat S$ one can see that the mappings $\hat R_{j}$ considered in \eqref{R_{j}} satisfy the open set condition. 

It follows from Theorem 9.3 in \cite{Fal} that the Hausdorff dimension $s:=\dim_{H}\hat E$ can be computed from the formula
\begin{equation}
\sum_{j=1}^{2n-1}(c_{j})^{s}=1,
\end{equation}
where $c_{j}$ are the ratios of the similarity maps $\hat R_{j}$. Therefore, in our case we have
\begin{equation}
\sum_{j=1}^{2n-1}\Bigl(\frac{1}{3^{k-1}}\Bigr)^{s}=1,
\end{equation}
and by an elementary computation we obtain
\begin{equation}
s=\frac{\log(2n-1)}{(k-1)\log 3}.
\end{equation}

Imposing now the condition $1<s<2$ we get 
\begin{equation}
\frac{3^{k-1}}{2}<n<\frac{3^{2(k-1)}+1}{2},
\end{equation}
where $n$ and $k$ are positive integers. A moment of thought convinces that 
\begin{equation}
n=\frac{3^{k-1}+3}{2}
\end{equation}
is a positive integer that satisfies this condition ($n$ is obviously integer because all powers of an odd number are odd and sum of two odd numbers is even). 

Therefore, we have
\begin{proposition}\label{prop 3.1}
Let $k\geq 2$ be a fixed integer and let $n:=\frac{3^{k-1}+3}{2}$. Then
\begin{equation}
\dim_{H}\hat E=\frac{\log(2n-1)}{(k-1)\log 3}\in (1,2),
\end{equation}
where $\hat E$ is the accumulation points set of $Y$ constructed as above.
\end{proposition}

Remark that the minimal admitted value for $n$ is $n=3$ for $k=2$. In this case we have 
$\dim_{H}\hat E=\frac{\log 5}{\log 3}=1.46497$. 

Let us point out that the sets $\hat S$ and $H=\cup_{m=0}^{\infty}\cup_{{j_{1}j_{2}\cdots j_{m}}}A_{{j_{1}j_{2}\cdots j_{m}}}\cup B_{{j_{1}j_{2}\cdots j_{m}}}$ are bi-Lipschitz, i.e. one can define a map $\Phi:\hat S\to H$ and there are positive constants $c$ and $C$ such that
\begin{equation}
c\cdot d(x,y)\leq d(\Phi(x),\Phi(y))\leq C\cdot d(x,y),
\end{equation}
for all $x,y\in \hat S$. Indeed, since $\phi$ has been chosen arbitrary, if we take a small $\phi<\varepsilon$, for any $\varepsilon>0$, then $\frac{\pi}{2}-\phi\to \frac{\pi}{2}$, i.e. the small spherical caps $S_{j_{1}j_{2}\cdots j_{m}}$ are almost included in the orthogonal planes to the rays $\gamma_{j_{1}j_{2}\cdots j_{m}}$. This means that $\hat S$ and $H$ are bi-Lipschitz.

On the other hand, it is known that the Hausdorff dimensions of two bi-Lipschitz sets coincide (see \cite{Fal} for example), and therefore we obtain
\begin{corollary}\label{cor 3.2}
Let $k\geq 2$ be a fixed integer and let $n:=\frac{3^{k-1}+3}{2}$. Then
\begin{equation}
\dim_{H} E=\frac{\log(2n-1)}{(k-1)\log 3}\in (1,2),
\end{equation}
where $E$ is the set of endpoints of $IT$ constructed in Section \ref{sec:2}.
\end{corollary}

We point out that the dimension $n(k)$ increases exponentially with $k$. 

\section{Proof of Theorem \ref{theorem A}}\label{sec:4}

In this section we construct a closed, convex ball in $\R^{n}$ such that the inward cut locus of $\partial D$ coincides with the infinite tree $IT$. Then we will smooth out the truncated cones 
$A_{j_{1}j_{2}\cdots j_{m}}$ at each depth level $m$.

Firstly, for each depth level $m$, let us consider an \lq \lq $\varepsilon$-dilatation'' of the closed, convex ball $H_{m}$, namely we define
\begin{equation}
D_{m}:=\{\textrm{the convex ball whose boundary is } \partial H_{m}\}\cup\{x\in\R^{n}\ |\ d(x,\partial H_{m})\leq \varepsilon\},
\end{equation}
for any positive constant $\varepsilon>0$. Denote the limit
\begin{equation}
D:=\lim_{m\to \infty}D_{m}
\end{equation}
that is the $\varepsilon$-dilatation of convex ball $H$. 

Note that the boundary $\partial D\in\R^{n}$ is the convex $C^{1}$-hypersurface 
\begin{equation}
\partial D=\lim_{m\to \infty}\Bigl[\cup_{i=0}^{m}\cup_{j_{1}j_{2}\cdots j_{m}}(\tilde A_{j_{1}j_{2}\cdots j_{m}}\cup \tilde B_{j_{1}j_{2}\cdots j_{m}})\Bigr]\cup \tilde P\cup \tilde E,
\end{equation}
where the tilde notation means the $\varepsilon$-dilatation of the 
corresponding geometrical objects of $H$.

One can easily see that the geodesics orthogonal to $\partial D$ are in fact geodesics to $H$ extended by $\varepsilon$ at one of their ends, and therefore the inward cut locus of $\partial D$ coincides with the infinite tree $IT$. Obviously the set of end points $E$ are now in the interior of the ball $D$ and not on its boundary as in the case of $H$. This allows us to realize the infinite tree $IT$ as the inward cut locus of a hypersurface in $\R^{n}$. 

\setlength{\unitlength}{1cm} 
\begin{center}
\begin{picture}(9, 9)

\includegraphics[height=9cm, angle=0]{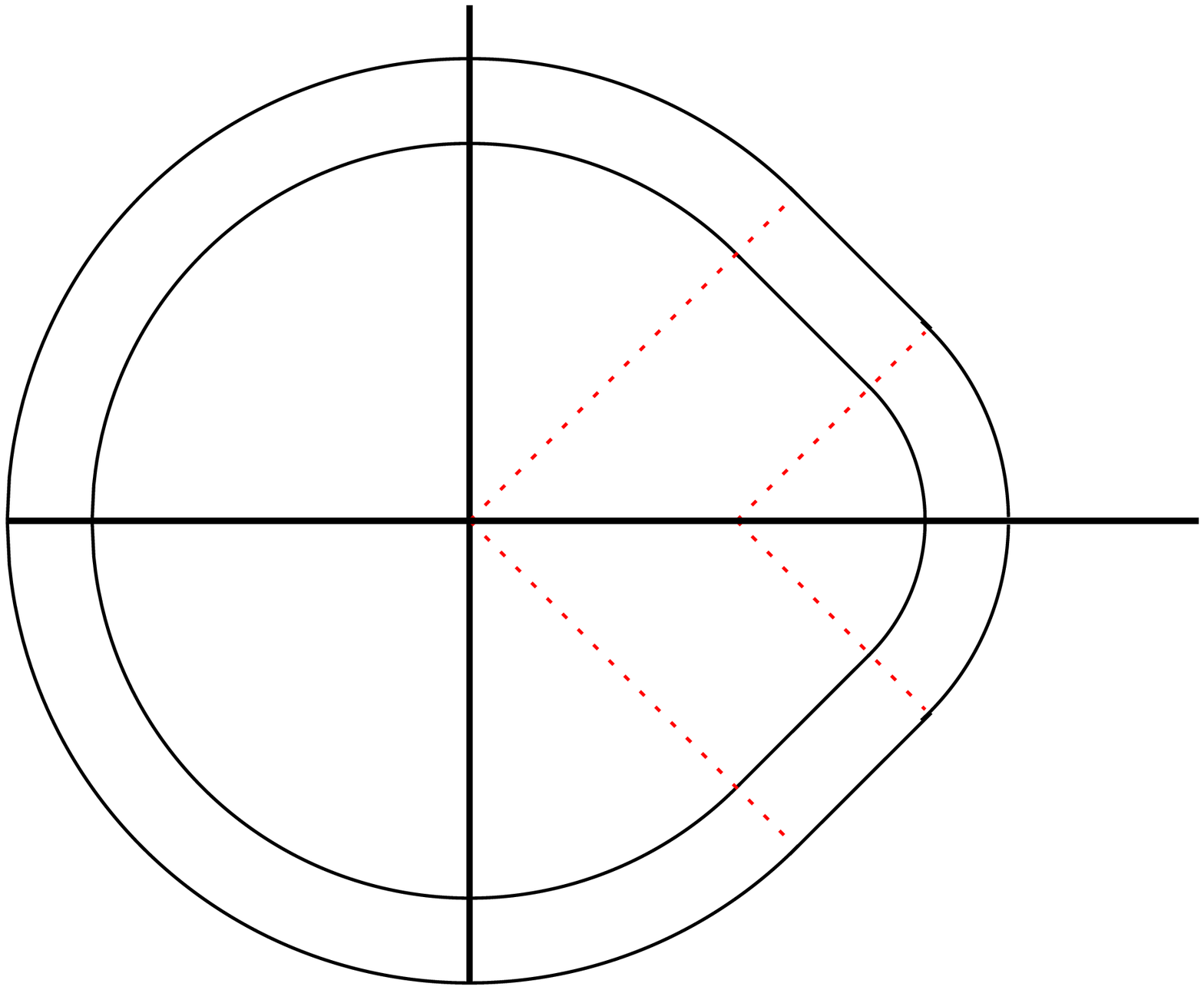}
\put(-5.8,4){$o$}
\put(-3.8,4){$q$}
\put(-8.6,4.5){$\varepsilon$}
\put(-6.5,6.5){$\partial H_{0}$}
\put(-8.5,7){$\partial D_{0}$}

\put(-3.8,5.7){$\hat c'$}
\put(-3,6.7){$\tilde c'$}

\put(-2.8,4.7){$\hat c$}
\put(-2,5.7){$\tilde c$}

\put(-9.3,3.8){$A$}
\put(-1.5,3.8){$B$}

\put(-0.3,4.3){\vector(1,0){0.1}}
\put(-0.5,4){$x_{1}$}
\put(-5.52,8){\vector(0,1){0.1}}
\put(-6.2,7.9){$x_{2}$}
\end{picture}
\end{center}

{\bf Figure 6.} The $\varepsilon$-dilatation of the convex ball $H_{0}$.


\bigskip

Secondly, we are going to smooth out the truncated cones $\tilde A_{j_{1}j_{2}\cdots j_{m}}$ on $\partial D$ such that the inward cut locus does not change. 

Let us refer to Figure 6 restricted to the upper half Euclidean plane 
$\mathbb H:=\{(x_{1},x_{2})\in \R^{2}\ |\ x_{2}>0\}$. Recall that $S$ and $S_{0}$ are circles of centers $o$ and $q$ with radii $r_{-1}$ and $r_{0}$, respectively. We further denote:

\begin{itemize}
 \item the intersection points of $\partial D_{0}$ with the horizontal coordinate axis by $A(a,0)$ and $B(b,0)$, $a<b$, respectively;
 \item the $x_{1}$ coordinates of the points $\tilde c'$ and $\tilde c$ by $t_{1}$ and $t_{2}$;
 \item the straightline segment determined by the points $\tilde c'$ and $\tilde c$ by $d=\{d(x_{1})\ |\ t_{1}<x_{1}<t_{2}\}$.
 \item the large circular segment on $\partial D_{0}$ from $A$ to $\tilde c'$ by $\{f_{1}(x_{1})\ |\ a<x_{1}<t_{1}\}$, and the small circular segment on $\partial D_{0}$ from $\tilde c$ to $B$ by $\{f_{2}(x_{1})\ |\ t_{2}<x_{1}<b\}$.
 \end{itemize}

We have the following smoothing Proposition:

\begin{proposition}\label{smoothing prop}
With the notations above, there is a $C^{\infty}$-function $F(x_{1})$ defined on $[a,b]$ that takes values in $\mathbb H$ such that 
\begin{itemize}
\item $F(x_{1})=f_{1}(x_{1})$ for $x_{1}\in [a,t_{1}]$;
\item $F(x_{1})=f_{2}(x_{1})$ for $x_{1}\in [t_{2},b]$;
\item $d(x_{1})\geq F(x_{1})\geq \max(f_{1}(x_{1}),f_{2}(x_{1}))$, for $t_{1}<x_{1}<t_{2}$;
\item all the inward straight lines (geodesic rays) orthogonal to $F(x_{1})$, for $t_{1}<x_{1}<t_{2}$, do not intersect each other in the region $\{x_{2}<d(x_{1})\}\cap \mathbb H$.
\end{itemize}
\end{proposition}

In other words, we will contract a smooth curve that joins points $\tilde c'$ and $\tilde c$. By the same operation in the lower half plane $\{(x_{1},x_{2})\ |\ x_{2}\leq 0\}$ we obtain a smooth plane curve.  


\setlength{\unitlength}{1cm} 
\begin{center}
\begin{picture}(10,10 )

\includegraphics[height=6cm, angle=0]{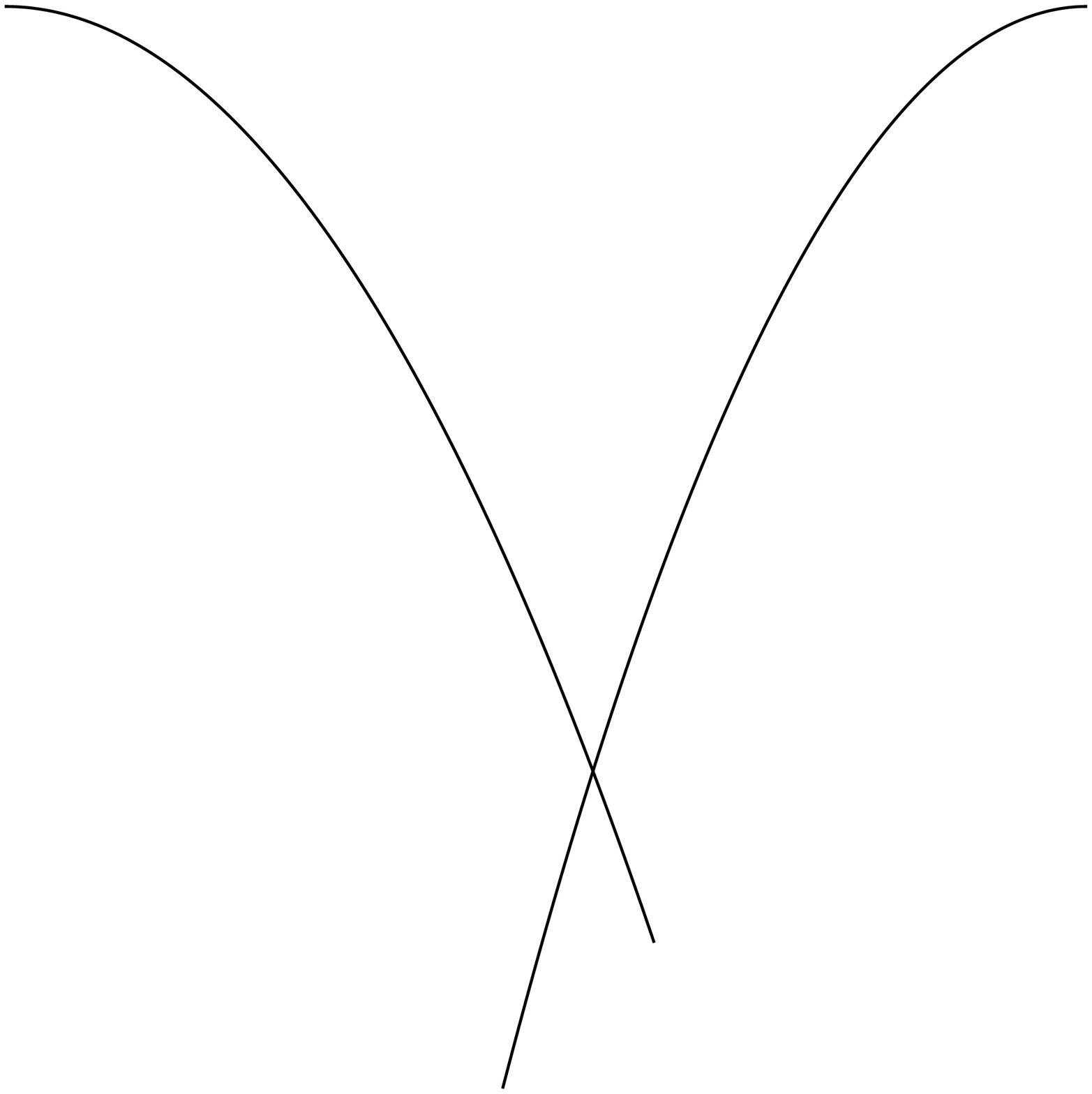}
\qquad\qquad
\includegraphics[height=6cm, angle=0]{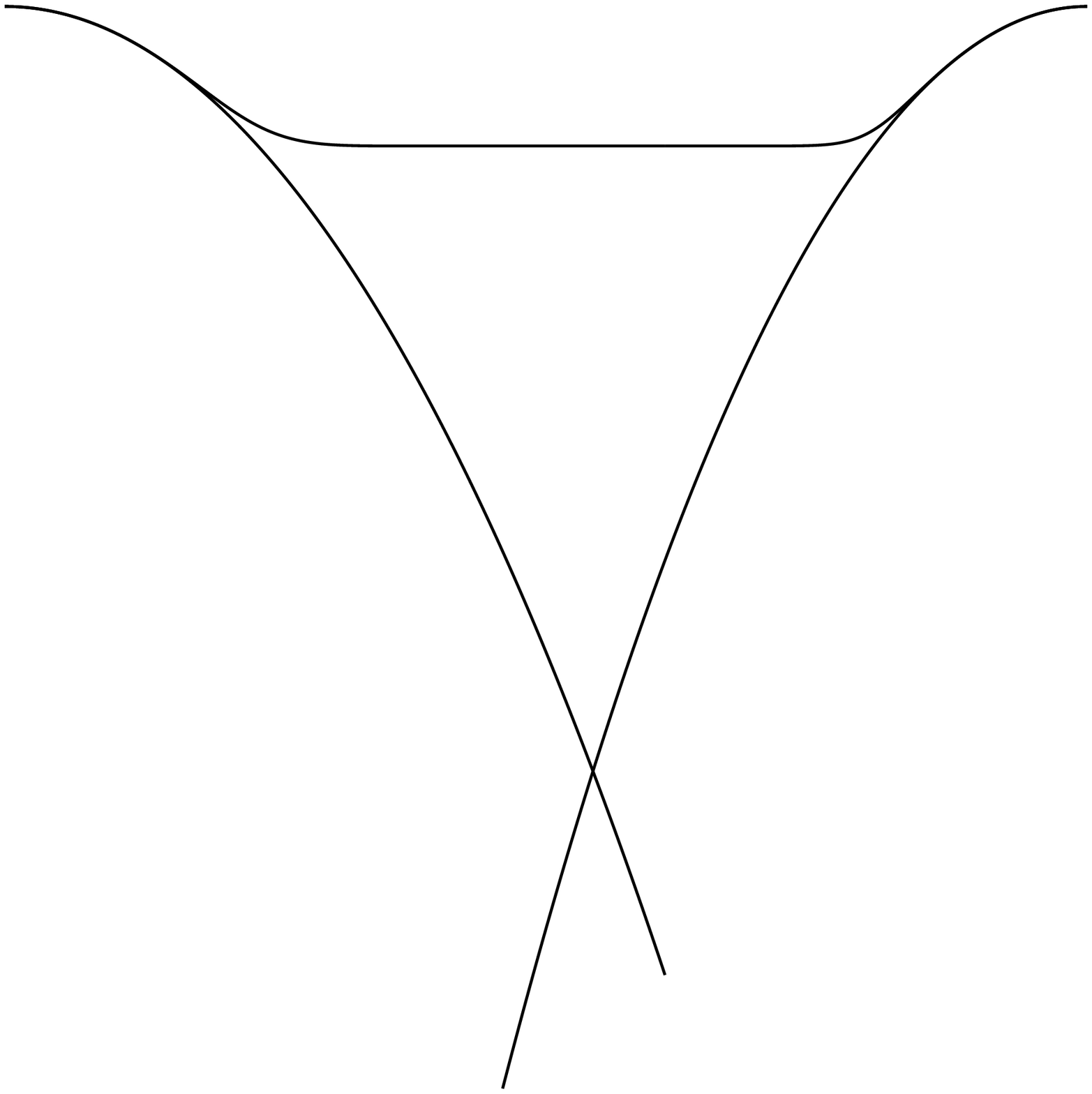}

\put(-13.7,0){\line(0,1){8}}
\put(-5.89,0){\line(0,1){8}}
\put(-13.7,8){\vector(0,1){0.1}}
\put(-5.89,8){\vector(0,1){0.1}}
\put(-13.7,5.9){\line(1,0){6}}
\multiput(-13.7,1.8)(0.5,0){13}{\line(1,0){0.2}}
\multiput(-13.7,5.16)(0.5,0){13}{\line(1,0){0.2}}
\multiput(-11.6,-2)(0,0.5){15}{\line(0,1){0.2}}
\multiput(-9.6,-2)(0,0.5){15}{\line(0,1){0.2}}
\put(-10.55,5.16){\circle*{0.1}}

\multiput(-10.55,-2)(0,0.5){15}{\line(0,1){0.2}}
\multiput(-5.89,5.9)(0.5,0){13}{\line(1,0){0.2}}
\multiput(-5.89,1.8)(0.5,0){13}{\line(1,0){0.2}}
\put(-13.7,-2){\line(1,0){6}}
\multiput(-13.7,-2)(0,0.2){13}{\line(0,1){0.05}}
\put(-5.89,-2){\line(1,0){6}}
\multiput(-5.89,-2)(0,0.2){13}{\line(0,1){0.05}}
\put(-7.7,-2){\vector(1,0){0.1}}
\put(0.11,-2){\vector(1,0){0.1}}
\put(-14.3,6){$\tilde c'$}
\put(-13.7,5.9){\circle*{0.1}}
\put(-8.3,6){$\tilde c$}
\put(-8,5.9){\circle*{0.1}}
\put(-14.3,8){$y$}

\put(-6.4,6){$\tilde c'$}
\put(-5.9,5.9){\circle*{0.1}}
\put(-0.3,6){$\tilde c$}
\put(-0.2,5.9){\circle*{0.1}}
\put(-6.4,8){$y$}

\put(-10.55,1.8){\circle*{0.1}}
\put(-11.6,5.16){\circle*{0.1}}
\put(-9.6,5.16){\circle*{0.1}}
\put(-9.02,5.16){\circle*{0.1}}
\put(-12.33,5.16){\circle*{0.1}}
\put(-12.5,4.5){$P$}
\put(-11.5,5.4){$Q$}
\put(-10.5,1.4){$R$}
\put(-9.6,5.4){$S$}
\put(-9,4.5){$T$}

\put(-11.6,-2.5){$x_Q$}
\put(-10.5,-2.5){$x_R$}
\put(-9.6,-2.5){$x_S$}

\put(-14,-2.4){$o$}
\put(-6,-2.4){$o$}
\put(-2.5,1.4){$R$}
\put(-2.76,1.8){\circle*{0.1}}

\put(-7.7,-2.5){$x$}
\put(0.11,-2.5){$x$}

\put(-14.3,1.7){$y_R$}
\put(-14.3,5.1){$y_b$}

\put(-11.3,0.5){$f_{2}$}
\put(-11.3,2.5){$f_{1}$}

\put(-3.5,0.5){$f_{2}$}
\put(-3.5,2.5){$f_{1}$}
\put(-10.3,6){$d$}
\put(-3,5.4){$F$}

\put(-10.55,-2){\circle*{0.1}}
\put(-11.6,-2){\circle*{0.1}}
\put(-9.6,-2){\circle*{0.1}}

\end{picture}
\end{center}
\vspace{3 cm}
{\bf Figure 7.} Smoothing of $\partial D_{0}$. Before smoothing (left), after smoothing (right).


\begin{proof}
For the sake of simplicity we rotate counterclockwise with $\frac{\pi}{2}-\phi$ the coordinates system used in Section 2 (compare with Figure 2) obtaining Figure 7 (left). The new coordinate system is denoted with $\{x,y\}$ while $\tilde c'$, $\tilde c$, $f_{1}$, $f_{2}$, $d$ have the same meaning as above. We denote the coordinates of points $\tilde c'$ and $\tilde c$ by $(x_{\tilde c'}, y_{\tilde c'})$ and 
$(x_{\tilde c}, y_{\tilde c})$, respectively. Moreover, we denote $\{R\}=f_{1}\cap f_{2}$ with coordinates
$(x_{R}, y_{R})$. In Figure 7 we have expressed by vertical dots the fact that $y_{R}$ is actually quite far from the origin $o$ and quite closed to $\tilde c'$, i.e. $d(o,y_{R})>>d(y_{R},y_{\tilde c'})$. 

With these notations, we consider the function 
\begin{equation}
F_{1}(x):=\int_{0}^{x}[1-g_{\sigma}(t)]f_{1}'(t)dt|_{\sigma=x}+y_{\tilde c'},\qquad 0\leq x\leq x_{R},
\end{equation}
where $0<\sigma\leq x_{R}$, 
\begin{equation}
g_{\sigma}(t):=
\begin{cases}
 0, &t \leq 0\\
 \frac{\varphi(t)}{\varphi(t)+\varphi(\sigma-t)}, &0<t<\sigma\\
 1, &t\geq \sigma
\end{cases}
\end{equation}
and
\begin{equation}
\varphi(t):=\begin{cases}
0, & t\leq 0\\
e^{-\frac{1}{t}}, & t>0.
\end{cases}
\end{equation}

\begin{center}
\includegraphics[height=7cm, angle=0]{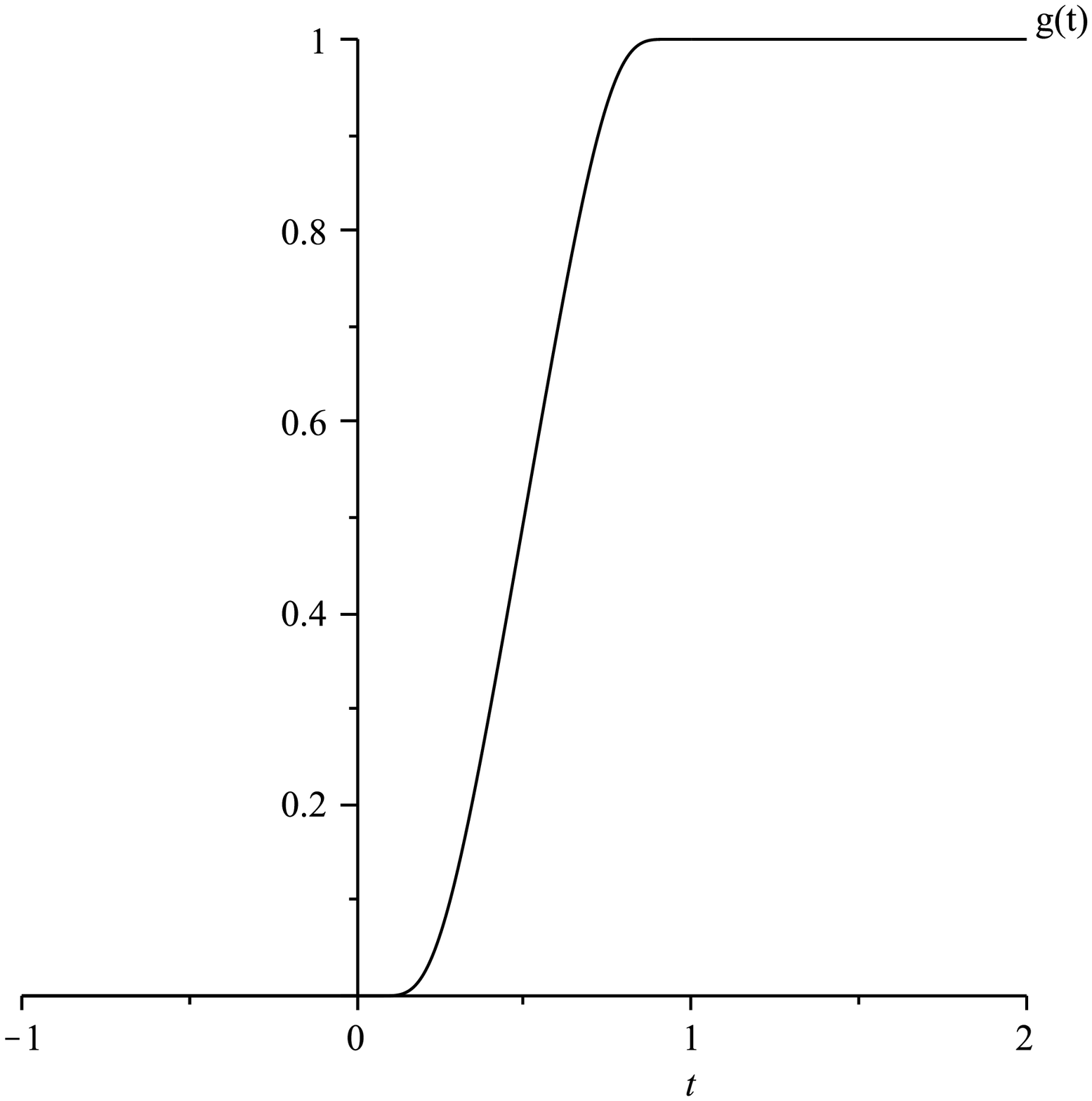}
\end{center}

{\bf Figure 8.} The graph of the function $g:=g_{\sigma}$ for $\sigma=1$.

\bigskip

It is elementary to see that $F_{1}$ is a smooth monotone non increasing function on $[0,x_{R}]$, 
$F_{1}(x)\geq f_{1}(x)$, $F_{1}(0)=x_{\tilde c'}$, $F_{1}(x_{R})=y_{R}+\delta_{1}$, $\delta_{1}>0$. Moreover $F_{1}'(0)=0=F_{1}'(x_{R})$.

Similarly, we define 
\begin{equation}
F_{2}(x):=\int_{x}^{1}h_{\rho}(t)f_{2}'(t)dt|_{\rho=x}+y_{\tilde c},\qquad x_{R}\leq x\leq x_{\tilde c},
\end{equation}
where $x_{R}<\sigma\leq x_{\tilde c}$, 
\begin{equation}
h_{\rho}(t):=
\begin{cases}
 1, &t \leq 0\\
 \frac{\varphi(t-\rho)}{\varphi(t-\rho)+\varphi(1-t)}, &0<t<\rho\\
 0, &t\geq \rho
\end{cases}
\end{equation}
and $\varphi$ same as above.

One can easily see that $F_{2}$ is a smooth monotone non decreasing function on 
$[x_{R}, x_{\tilde c}]$, 
$F_{2}(x)\geq f_{2}(x)$,  $F_{2}(x_{R})=y_{R}+\delta_{2}$, $\delta_{2}>0$, $F_{2}(x_{\tilde c})=y_{\tilde c}$. Moreover $F_{2}'(x_{R})=0=F_{2}'(x_{\tilde c})$.

It is obvious that we can choose now a constant 
$y_{b}\in (\max\{y_{R}+\delta_{1}, y_{R}+\delta_{2}\}, y_{\tilde c})$, in practice $y_{b}$ will be taken as closed as possible to $y_{\tilde c'}$, such that there exists $x_{Q}\in (0,x_{R})$ and $x_{S}\in (x_{R}, x_{\tilde c})$ that satisfies $F_{1}(x_{Q})=F_{2}(x_{S})=y_{b}$ (see Figure 7).


\setlength{\unitlength}{1cm} 
\begin{center}
\begin{picture}(7,7 )

\includegraphics[height=6cm, angle=0]{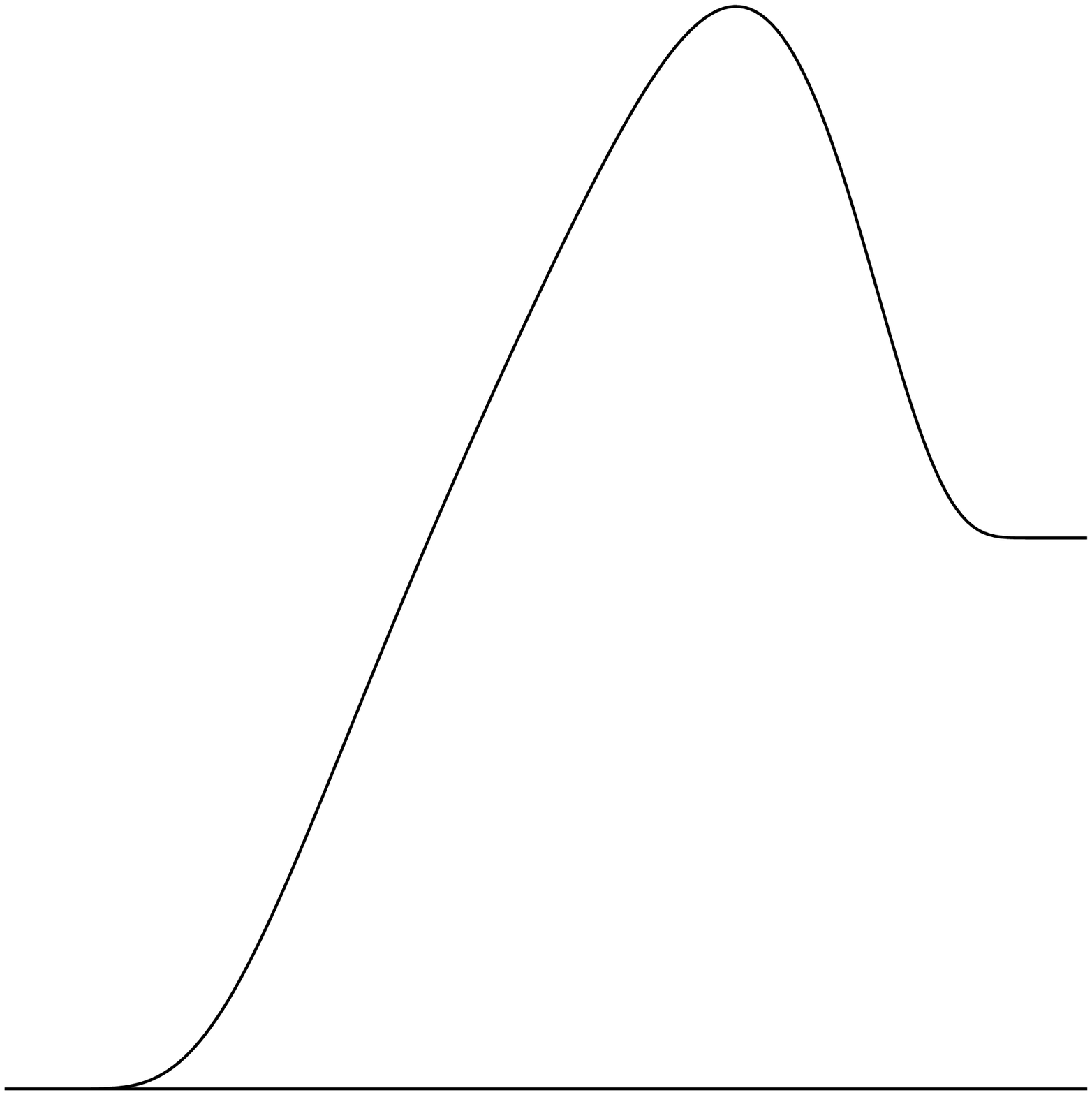}

\put(-5.9,3.05){\line(1,0){8}}
\put(-5.9,6.05){\vector(0,1){0.1}}

\put(-5.9,0){\line(0,1){6}}
\put(2,3.05){\vector(1,0){0.1}}

\put(-6.3,3){$o$}
\put(2,2.6){$x$}

\put(0,2.6){$x_{Q}$}

\put(-3,0.3){$k_{f_{1}}(x)$}

\put(-3,6){$k_{F_{1}}(x)$}

\put(-7.3,0){$-\frac{1}{\sqrt{r_{-1}}}$}

\multiput(-0.1,0.1)(0,.5){13}{\line(0,1){0.1}}
\end{picture}
\end{center}
\vspace{1 cm}
{\bf Figure 9.} The curvature of $F_{1}(x)$. 
\bigskip


With these, we redefine 
\begin{equation}
F_{1}(x):=\int_{0}^{x}[1-g_{x_{Q}}(t)]f_{1}'(t)dt+y_{\tilde c'},\qquad 0\leq x\leq x_{Q},
\end{equation}
and 
\begin{equation}
F_{2}(x):=\int_{x}^{1}h_{x_{S}}(t)f_{2}'(t)dt+y_{\tilde c},\qquad x_{S}\leq x\leq x_{\tilde c}.
\end{equation}

Then the function 
\begin{equation}
F(x):=\begin{cases}
F_{1}(x), & 0\leq x\leq x_{Q}\\
y_{b}, & x_{Q}<x<x_{S}\\
F_{2}(x), & x_{S}\leq x\leq x_{\tilde c}
\end{cases}
\end{equation}
will have the required properties. 

One can see that the curvature of $F_{1}$, namely $k_{F_{1}}(x):=\frac{F_{1}''}{[1+(F_{1}')^{2}]^{3/2}}$, 
satisfies $k_{f_{1}}\leq k_{F_{1}}$ on $[0,x_{\tilde c}]$ (see Figure 9) and similar for $F_{2}$. This guaranties that $F$ smoothly joins points $\tilde c'$ and $\tilde c$ and that the inner normals to $F$ cannot  intersect each other in the region $[0,x_{\tilde c}]\cap \mathbb H$.

$\qedd$
\end{proof}

We denote by $\tilde D_{0}$ the ball obtained from $D_{0}$ after smoothing out the truncated cones making use of the function $F$ introduced in Proposition \ref{smoothing prop}.  $D_{0}$ is actually a surface of revolution obtained by rotating the profile curve $F_{1}([a,t_{1}])\cup
F_{2}([t_{2},b])\cup d((t_{1},t_{2}))$ around the $\gamma$ axis. This is a $C^{\infty}$-surface whose inward cut locus is the segment $oq$.

This construction can be repeated in any dimension, we apply the above smoothing procedure to each $\partial D_{m}$, inductively. In this way we end up with a ball $\widetilde D_{m}\subset \R^{n}$ with $C^{\infty}$ boundary. 

By putting 
\begin{equation}
\widetilde D :=\lim_{m\to \infty} \widetilde D_{m}
\end{equation}
we obtain a closed, convex ball in $\R^{n}$ with boundary $\partial \widetilde D$ such that the inward cut locus of $\partial \widetilde D$ is the infinite tree $IT$, actually a fractal as shown in Section 3. 

One should remark that the set of end points $E$ of the fractal $IT$ are in the interior of $\widetilde D$ at distance $\varepsilon$ from the hypersurface $\partial \widetilde D$. However, taking into account that these end points actually lie on the tree branches $\gamma_{j_{1}\dots j_{m}}$, by extending these branches, they will intersect $\partial \widetilde D$ giving a set of points $\widetilde E$ on $\partial \widetilde D$. Such a point is the point $B$ in FIgure 6, for $m=0$). We point out that the points $\widetilde E$ are on $\partial \widetilde D$, but they do not belong to the fractal $IT$, they are only an $\varepsilon$-displacement of the end points $E$.

We will study in the following the differentiability of $\partial \widetilde D$. We have
\begin{proposition}\label{deg of diff}
The hypersurface $\partial \widetilde D$ in $\R^{n}$ is:
\begin{enumerate}
\item at least $k$-differentiable at points $\widetilde E$,
\item $C^{\infty}$ at any point $\partial \widetilde D\setminus \widetilde E$.
\end{enumerate}
\end{proposition}

\begin{proof}
Let us denote
\begin{equation}
\widetilde q_{j_{1}j_{2}\dots j_{m}}:=\partial \widetilde D_{m}\cap \gamma_{j_{1}j_{2}\dots j_{m}},\qquad 
\widetilde Q:=\bigcup_{m=1}^{\infty}\{q_{j_{1}j_{2}\dots j_{m}}\}.
\end{equation}

It can be seen that $\widetilde E$ is the set of accumulation points of $\widetilde Q$. A moment of thought shows that the set $\partial \widetilde D_{m}\cap \partial \widetilde D_{m-1}$ is a circle on the sphere $S(q_{j_{1}j_{2}\dots j_{m-1}},r_{m-1}+\varepsilon)$, namely the base of the spherical cap cut off by the smoothed truncated cone $\widetilde A_{j_{1}j_{2}\dots j_{m}}$, and that $\widetilde q_{j_{1}j_{2}\dots j_{m-1}}=S(q_{j_{1}j_{2}\dots j_{m-1}},r_{m-1}+\varepsilon)\cap \gamma_{j_{1}j_{2}\dots j_{m}}$. Obviously $\widetilde E$ is not countable.

We define the function 
\begin{equation}
\zeta:\bigcup_{m=1}^{^\infty}\partial \widetilde{D}_{m-1}\to (0,\infty),\quad \zeta(\widetilde q_{j_{1}j_{2}\dots j_{m-1}})=
d(\widetilde q_{j_{1}j_{2}\dots j_{m}\dots \infty},\widetilde q_{j_{1}j_{2}\dots j_{m-1}}),
\end{equation}
where $\widetilde q_{j_{1}j_{2}\dots j_{m}\dots \infty}\in \widetilde E$ (see Figure 10).

One can see that the differentiability of $\partial \widetilde D$ at points of $\widetilde E$ can be expressed in terms of the differentiability of $\zeta$. Indeed, let $A\in \partial \widetilde D$ be a point such that there is an $m$ for which $A\in \partial \widetilde D_{m-1}$  and denote by $\mathfrak r(u^{1},u^{2},\dots,u^{n-1})$ its position vector in $\R^{n}$, when we denote by $(u^{1},u^{2},\dots,u^{n-1})$ the local coordinates on  $\widetilde D$. Likely, let $B\in \partial \widetilde D$ be a point such that $B\in \partial \widetilde D_{m}$ with position vector $\Phi(u^{1},u^{2},\dots,u^{n-1})\in \R^{n}$. Then it is clear that
\begin{equation}\label{position vectors}
\Phi(u^{1},u^{2},\dots,u^{n-1})=h(u^{1},u^{2},\dots,u^{n-1})\cdot \mathfrak e(u^{1},u^{2},\dots,u^{n-1})+
\mathfrak r(u^{1},u^{2},\dots,u^{n-1}),
\end{equation}
where $\mathfrak e(u^{1},u^{2},\dots,u^{n-1})$ is the outward pointing unit normal vector to $\partial \widetilde D_{m-1}$ at $A$, and $h$ is the height function $h(u^{1},u^{2},\dots,u^{n-1})=d(A,B)$. Here $d$ is the usual Euclidean distance. Then we can see that $h|_{\widetilde E}=\zeta |_{\widetilde E}$.

Obviously $\Phi$ is $C^{\infty}$ at any point $\widetilde D\setminus \widetilde E$ by construction, so we need to investigate only the differentiability of $\zeta |_{\widetilde E}$.

Let us recall that the ${r}$ differential of $\zeta$ can be written using finite differences as follows
\begin{equation}
\frac{d^{r}\zeta}{dx^{r}}=\lim_{h\to 0} \frac{\delta^{r}_{h}[\zeta]}{h^{r}},\quad r\geq 1
\end{equation}
where 
\begin{equation}
\delta^{r}_{h}[\zeta]:=\sum_{i=0}^{r}(-1)^{i}
\begin{pmatrix}
r\\i
\end{pmatrix}
\zeta\Bigl(x+\Bigl(\frac{r}{2}-i\Bigr)h\Bigr)
\end{equation}
and $h=d(\partial \widetilde D_{m}\cap \partial \widetilde D_{m-1},\widetilde q_{j_{1}j_{2}\dots j_{m-1}})$.


\setlength{\unitlength}{1cm} 
\begin{center}
\begin{picture}(10, 10)
\includegraphics[height=10cm, angle=0]{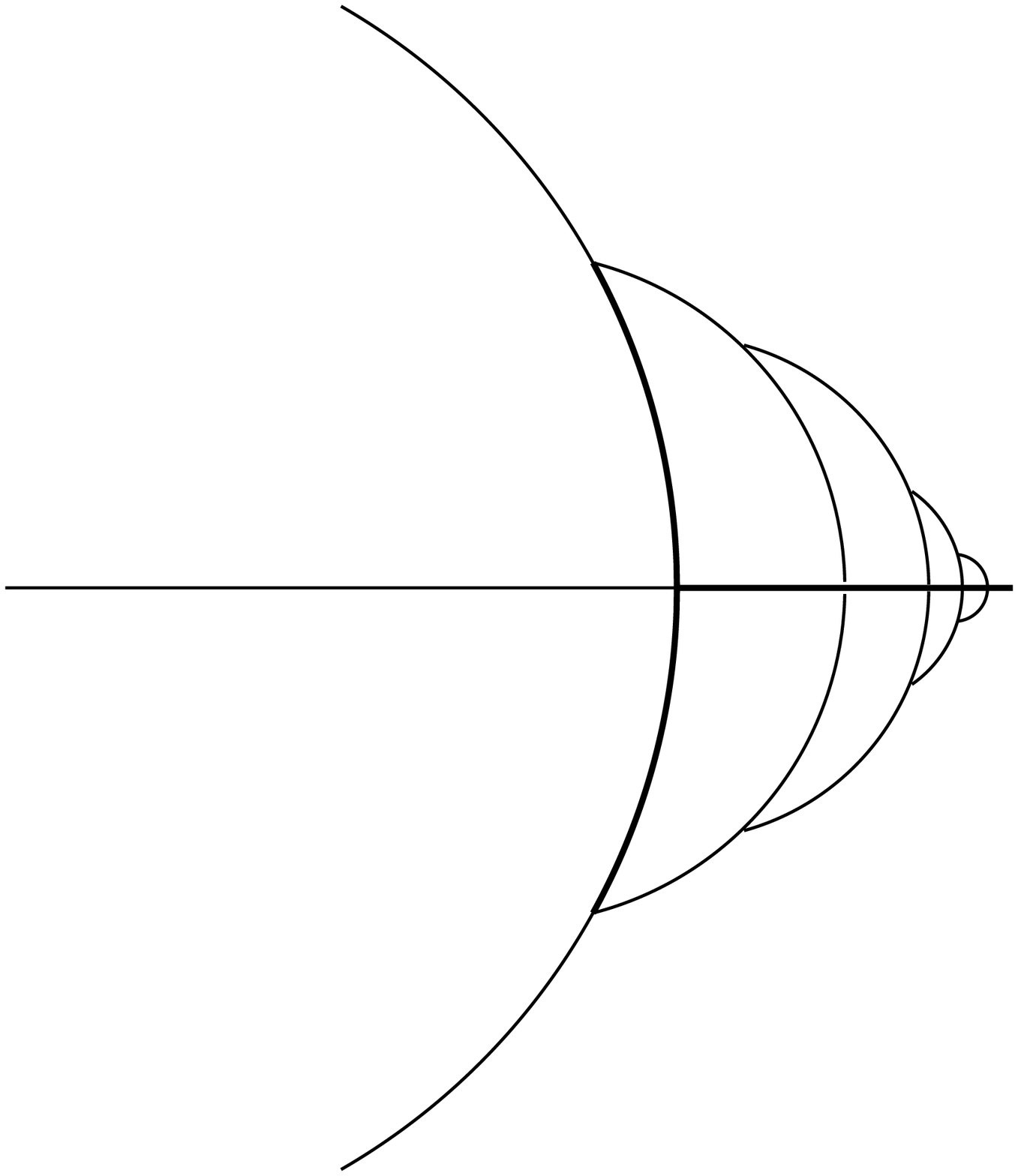}
\put(-0.5,5){$\widetilde e\in \widetilde E\subset \partial \widetilde D$}
\put(-9,5.3){$\gamma$}
\put(-7.5,0.3){$\partial \widetilde{D}_{m-1}$}
\put(-4,2){$\partial \widetilde{D}_{m}$}

\put(-4,4.5){$\widetilde q$}
\put(-0.8,5){\circle*{0.1}}
\put(-3.6,5){\circle*{0.1}}
\put(-3.3,5.3){$\zeta(\widetilde q)$}

\put(-4,6){$h$}
\end{picture}
\end{center}

{\bf Figure 10.} The function $\zeta$. For simplicity we have denoted $\gamma:=\gamma_{j_{1}j_{2}\dots j_{m-1}}$,  $\widetilde q:=\widetilde q_{j_{1}j_{2}\dots j_{m-1}}$ and $\widetilde e:=\widetilde q_{j_{1}j_{2}\dots j_{m}\dots \infty}$ (the smoothing function $F$ is not drawn). 


\bigskip

Let us remark that 
\begin{equation}\label{finite diff ratio}
\begin{split}
\frac{\delta^{r}_{h}[\zeta]}{h^{r}} & \leq \frac{2^{k-1}}{h^{r}} d(\max \zeta,\min \zeta)=
\frac{2^{k-1}}{h^{r}} d(\widetilde q_{j_{1}j_{2}\dots j_{m}\dots \infty},\widetilde q_{j_{1}j_{2}\dots j_{m-1}})\\
& =\frac{2^{k-1}}{h^{r}} \sum_{i=0}^{\infty}d(\widetilde q_{j_{1}j_{2}\dots j_{m}},\widetilde q_{j_{1}j_{2}\dots j_{m-1}})
\end{split}
\end{equation}

On the other hand, we compute
\begin{equation}\label{denom}
h=d(\partial \widetilde D_{m}\cap \partial \widetilde D_{m-1},\widetilde q_{j_{1}j_{2}\dots j_{m-1}})=
\varphi\Bigl(\frac{1}{3}\Bigr)^{m-1}(r_{m-1}+\varepsilon)>\phi\Bigl(\frac{1}{3}\Bigr)^{m}(r_{m-1}+\varepsilon),
\end{equation}
and 
\begin{equation}\label{numer}
\begin{split}
d(\widetilde q_{j_{1}j_{2}\dots j_{m}},\widetilde q_{j_{1}j_{2}\dots j_{m-1}})& =(r_{m}+\varepsilon)-(r_{m-1}+\varepsilon)+l_{m}=l_{m}\Bigl(1-\cos\phi\Bigl(\frac{1}{3}\Bigr)^{m}\Bigr)\\
& =
\frac{1}{(3^{k-1})^{m}}\tan\Bigl(\frac{1}{2}\phi\Bigl(\frac{1}{3}\Bigr)^{m}\Bigr),
\end{split}
\end{equation}
where we have used $r_{m}-r_{m-1}=-l_{m}\cos\phi\Bigl(\frac{1}{3}\Bigr)^{m}$, and the well known trigonometric formula $\dfrac{1-\cos\eta}{\sin\eta}=\tan\frac{\eta}{2}$.

One can easily see that there are infinitely many terms in the sum in the right hand side of the last equality in \eqref{finite diff ratio}, therefore, there is no harm in leaving out a finite number of them, say the first $m-1$ terms. It follows
\begin{equation}
\begin{split}
d(\widetilde q_{j_{1}j_{2}\dots j_{m}\dots \infty},\widetilde q_{j_{1}j_{2}\dots j_{m-1}})
=\sum_{i=m}^{\infty} \frac{1}{(3^{k-1})^{i}}\tan\Bigl(\frac{1}{2}\phi\Bigl(\frac{1}{3}\Bigr)^{i}\Bigr)
=\sum_{i=m}^{\infty} \frac{1}{(3^{k-1})^{i}}\Bigl(\frac{1}{2}\phi\Bigl(\frac{1}{3}\Bigr)^{i}\Bigr)
=\frac{\phi}{2}\frac{1}{3^{{km}-1}}
\end{split}
\end{equation}
where we have used the well-known formula $\lim_{\tau\to 0}\dfrac{\tan\tau}{\tau}=1$.

Let us remark now that $h\to 0$ when $m\to \infty$ by construction, so there is no harm in regarding $\lim_{h\to 0}$ as $\lim_{m\to \infty}$. 

It results
\begin{equation}\label{big lim}
\lim_{h\to 0}\frac{\delta^{r}_{h}[\zeta]}{h^{r}} \leq \lim_{m\to \infty}\frac{\frac{\phi}{2}\frac{1}{3^{{km}-1}}}{\Bigl[\phi\Bigl(\frac{1}{3}\Bigr)^{m}(r_{m-1}+\varepsilon)\Bigr]^{r}}
=\frac{1}{2}\lim_{m\to \infty}\Bigl[\frac{1}{\phi^{r-1}(r_{m-1}+\varepsilon)^{r}}\times \frac{1}{3^{mk-r(m-1)}}\Bigr]
\end{equation}

One can now easily see that in the case $r=k$ the limit in the right hand side of \eqref{big lim} takes the finite value $\dfrac{1}{2}\dfrac{1}{3^{k}}\dfrac{1}{\phi^{k-1}\varepsilon^{k}}$ and therefore $\lim_{h\to 0}\frac{\delta^{r}_{h}[\zeta]}{h^{r}}$ is finite at a point of $\widetilde E$, hence
the function $\zeta$ is at least $k$-differentiable on $\widetilde E$. Remark that for $r<k$,  \eqref{big lim} implies $\dfrac{d^{r}\zeta}{dx^{r}}=0$, namely $\partial \widetilde D$ is $C^{k-1}$ everywhere,
while for $r>k$ we cannot say anything about the convergence of the limit in left hand side of \eqref{big lim}.

The proposition is proved.

$\qedd$
\end{proof}


In order to finish our construction, we adapt an idea of A. Weinstein from \cite{W}, namely the technique of {\it making any disc a unit disc}. We have
\begin{proposition}\label{Ck-Weinstein}
Let $\Delta$ be an $n$-dimensional ball embedded in a  $C^{r}$ manifold $M$ of dimension $n$. For any Riemannian metric on $M\setminus (interior\ of\ \Delta)$, there is a new Riemannian  on $M$, agreeing with the original metric on $M\setminus (interior\ of\ \Delta)$, such that for some point $p\in \Delta$ the exponential map $exp_{p}$ is a $C^{r}$ diffeomorphism of the unit ball around the origin in $T_{p}M$ onto $\Delta$.
\end{proposition}

This proposition can be proved by exactly the same method as Theorem C in \cite{W}.

We reach the final stage of our construction. 

Recall that an $n$-sphere $\Sph ^{n}$ can be always constructed topologically by gluing together the boundaries of a pair of $n$-balls $(\Delta_{1},\Delta_{2})$ provided they have opposite orientations. The boundary of an $n$-ball $\Delta_{i}$ is an $(n-1)$-sphere $\partial \Delta_{i}\simeq \Sph^{n-1}$, and these two $(n-1)$-spheres are to be identified, for $i=1,2$.

In other words, if we have a pair of $n$-balls of the same size, we superpose them so that their $(n-1)$-spherical boundaries match, and let matching pairs of points on the pair of $(n-1)$-spheres be identically equivalent to each other $\partial \Delta_{1}\equiv \partial \Delta_{2}$. In analogy with the case of the 2-sphere, the gluing surface, that is an $(n-1)$-sphere subset of $\Sph^{n}$, can be called \lq\lq equatorial sphere''. Obviously, the interiors of the original $n$-balls are not glued to each other but they cover the ``exterior'' surface of $\Sph^{n}=\Delta_{1}\bigcup \Delta_{2}$. 

Keeping this topological construction in mind, we will glue together the $n$-ball $\widetilde D$ constructed above with a new $n$-dimensional ball $\Delta$ by identifying $\partial \Delta$ with the hypersurface $H$ through a $C^{k}$ diffeomorphism. By this construction we obtain an $n$-sphere $\Sph^{n}$ whose equatorial sphere is $\partial \Delta\equiv \partial \widetilde D\simeq \Sph^{n-1}$. Moreover, the infinite tree $IT\subset \widetilde D$ lies down on the surface $\Sph^{n}$, more precisely, in the open region cut off by the equatorial sphere and covered by the interior of $\widetilde D$. By construction, all the cut points of $IT$ are at a certain distance (the $\varepsilon$-dilatation) from the  equatorial sphere. Therefore, $IT\subset \Sph^{n}$, but $IT \cap \bar{\Delta}=\emptyset$, where  $\bar{\Delta}$ is the closure of $\Delta$. 

Finally, by applying the Proposition \ref{Ck-Weinstein} to this $n$-sphere $M=\Sph^{n}$, in the case $r=k-1$ with the supplementary condition of $k$-differentiability (see Proposition \ref{deg of diff}), and asking $n=\frac{3^{k+1}}{2}+1$,
we obtain a new Riemannian metric $g$ on $\Sph^{n}$, metric that coincides on $\Sph^{n}\setminus (interior\ of\  \Delta)$ with the initial flat metric defined on $\widetilde D$,  and a point $p\in (interior\ of\  \Delta)\subset \Sph^{n}$ whose cut locus is $IT$. Obviously, the geodesics of $(M,g)$ starting from $p$ must cross the equatorial sphere before hitting the cut locus $IT$.

Of course, this Riemannian structure can not be $C^{\infty}$ because this would contradict with the main result in \cite{IT}, namely that the Hausdorff dimension of the cut locus of any point on a $C^{\infty}$ Riemannian manifold must be an integer.  

Also we remark that $k\to \infty$ would lead to $n\to \infty$ and hence our $n$-sphere must be infinite dimensional, but this is not allowed (see Proposition \ref{prop 3.1} and Corollary \ref{cor 3.2}).


The Theorem \ref{theorem A} is now proved.

\section{Randers metrics: a ubiquitous family of Finsler structures}\label{sec:5}

Let us recall that  a {\it Finsler manifold} $(M,F)$ is a $n$-dimensional differential manifold $M$ endowed with a norm 
$F:TM\to [0,\infty)$ such that
\begin{enumerate}
\item $F$ is positive and differentiable;
\item $F$ is 1-positive homogeneous, i.e. $F(x,\lambda y)=\lambda F(x,y)$, $\lambda>0$, $(x,y)\in TM$;
\item the Hessian matrix $g_{ij}(x,y):=\dfrac{1}{2}\dfrac{\partial^2 F^2}{\partial y^i\partial y^j}$ is positive definite on $\widetilde{TM}:=TM\setminus\{0\}$.
\end{enumerate}

The Finsler structure is called {\it absolute homogeneous} if $F(x,-y)=F(x,y)$ because this leads to the homogeneity condition $F(x,\lambda y)=|\lambda| F(x,y)$, for any $\lambda\in \mathbb R$. 

By means of the Finsler fundamental function $F$ one defines the {\it indicatrix bundle} (or the Finslerian {\it unit sphere bundle}) by $SM:=\cup_{x\in M}S_xM$, where $S_xM:=\{y\in M\ :\ F(x,y)=1\}$. 

On a Finsler manifold $(M,F)$ one can easily define the integral length of curves as follows. Let $\gamma:[a,b]\to M$ be a regular piecewise $C^{\infty}$ curve in $M$, and let $a:=t_0<t_1< \dots < t_k:=b$ be a partition of $[a,b]$ 
such that $\gamma_{|[t_{i-1},t_i]}$ is smooth for each interval $[t_{i-1},t_i]$, $i\in\{1,2,\dots,k\}$. The 
{\it integral length} of $\gamma$ is given by
\begin{equation}\label{integral length}
L(\gamma):=\sum_{i=1}^k\int_{t_{i-1}}^{t_i}F(\gamma(t),\dot\gamma(t))dt,
\end{equation}
where $\dot\gamma=\dfrac{d\gamma}{dt}$ is the tangent vector along the curve $\gamma_{|[t_{i-1},t_i]}$.

For such a partition, let us consider the regular piecewise $C^\infty$ map
\begin{equation}
\bar \gamma:(-\varepsilon,\varepsilon)\times[a,b]\to M,\quad (u,t)\mapsto \bar\gamma(u,t)
\end{equation}
such that $\bar\gamma_{|(-\varepsilon,\varepsilon)\times [t_{i-1},t_i]}$ is smooth for all $i\in\{1,2,\dots,k\}$, and
$\bar\gamma(0,t)=\gamma(t)$. Such a curve is called a regular piecewise $C^\infty$ {\it variation} of the base curve $\gamma(t)$, and the vector field 
$U(t):=\dfrac{\partial\bar\gamma}{\partial u}(0,t)$ is called the {\it variational vector field} of $\bar\gamma$. The integral length of $\bar\gamma(u,t)$ will be a function of $u$, defined as in \eqref{integral length}.

By a straightforward computation 
one obtains
\begin{equation}\label{arbitrary first variation}
\begin{split}
L'(0)=& g_{\dot\gamma(b)}(\gamma,U)|_a^b+\sum_{i=1}^k\Bigl[g_{\dot\gamma(t_i^-)}(\dot\gamma(t_i^-),U(t_i))-
g_{\dot\gamma(t_i^+)}(\dot\gamma(t_i^+),U(t_i))\Bigl]\\
& -\int_a^bg_{\dot\gamma}(D_{\dot\gamma}{\dot\gamma},U)dt,
\end{split}
\end{equation}
where $D_{\dot\gamma}$ is the covariant derivative along $\gamma$ with respect to the Chern connection and $\gamma$ is arc length parametrized 
(see \cite{BCS}, p. 123, or \cite{S}, p. 77 for details of this computation as well as for the basis on Finslerian connections).  

A regular $C^\infty$ piecewise curve $\gamma$ on a Finsler manifold is called a {\it geodesic} if $L'(0)=0$ for all piecewise 
$C^\infty$ variations of $\gamma$ that keep its end points fixed.  
In terms of Chern connection a unit speed geodesic is characterized by the condition $D_{\dot\gamma}{\dot\gamma}=0$.

Using the integral length of a curve, one can define the Finslerian distance between two points on $M$. For any two points $p$, $q$ on $M$, let us denote by $\Omega_{p,q}$ the set of all piecewise $C^\infty$ curves $\gamma:[a,b]\to M$ such that $\gamma(a)=p$ and $\gamma(b)=q$. The map
\begin{equation}
d:M\times M\to [0,\infty),\qquad d(p,q):=\inf_{\Omega_{p,q}}L(\gamma)
\end{equation}
gives the {\it Finslerian distance} on $M$. It can be easily seen that $d$ is in general a quasi-distance, i.e. it has the properties
\begin{enumerate}
\item $d(p,q)\geq 0$, with equality if and only if $p=q$;
\item $d(p,q)\leq d(p,r)+d(r,q)$, with equality if and only if $p$, $q$, $r$ are on the same geodesic segment (triangle inequality).
\end{enumerate}

In the case when $(M,F)$ is absolutely homogeneous, the symmetry condition $d(p,q)=d(q,p)$ holds and therefore $(M,d)$ is a genuine metric space. We do not assume this symmetry condition in the present paper.

A Randers metric on an $n$-differential manifold $M$ is a special Finsler metric $(M,F:=\alpha+\beta)$ obtained obtained by a deformation of a Riemannian metric $a:=a_{ij}(x)dx^{i}\otimes dx^{j}$ by a one-form $\beta:=b_{i}(x)dx^{i}$. The resulting Finslerian norm $F:TM\to [0,\infty)$ is given by
\begin{equation}
F(x,y):=\alpha(x,y)+\beta(x,y)=\sqrt{a_{ij}(x)y^{i}y^{j}}+b_{i}(x)y^{i}, \quad (x,y)\in TM.
\end{equation}

It is well known that imposing the condition $b^{2}:=a(b,b)<1$ is enough to ensure that this $F$ is a positive definite Finsler structure in the usual sense (see \cite{BCS} for details).

The geodesic equations of $F$ can be expressed in terms of $\alpha$ and $\beta$, but we don't need to do this. 

The following results are well known (see \cite{BCS}):
\begin{proposition}
Let $(M,F=\alpha+\beta)$ be a Randers space. The 1-form $\beta$ is closed if and only if the geodesics of the Randers metric $F$ coincide with the geodesics of the underlying Riemannian structure as point sets.
\end{proposition}

In other words, the Randers space $(M,F=\alpha+\beta)$ has reversible geodesics (see \cite{SS} for details). 

\begin{corollary}
Let $(M,\alpha)$ be a Riemannian space form and $\beta$ a 1-form on $M$. The 1-form $\beta$ is closed if and only if the Randers metric $(M,F=\alpha+\beta)$ is projectively flat. 
\end{corollary}

\begin{remark}
The Randers metrics can be described as the deformation of the Riemannian metric $a$ by means of a magnetic field specified by the 1-form $\beta$.

\end{remark}
\subsection{A Randers metric on the $n$-ball $\mathcal B^{n}$}

Let us consider the 2-ball $\mathcal B^{2}:=\{(x,y)\ |(x,y)|\leq 1\}\subset \R^{2}$, where $|\cdot|$ is the usual Euclidean norm, and introduce polar coordinates $(r,\theta)$, where $r=\sqrt{x^{2}+y^{2}}$ and 
\begin{equation}
x=r\cos\theta \qquad y=r\sin\theta.
\end{equation}

We will denote the coordinates of that tangent space at a point to $\mathcal B^{2}$ by 
$(r,\theta;x,y)\in T\mathcal B^{2}$ regarding $(r,\theta)$ as coordinates on the base manifold $\mathcal B^{2}$ and $(x,y)\in T_{(r,\theta)}\mathcal B^{2}$ the fiber coordinates.

The inward geodesics with the start point on the 2-ball $\mathcal B^{2}$ endowed with the Euclidean metric are rays that gather in the origin $O\in \R^{2}$ and since are all Euclidean unit length, the inward cut locus of the boundary $\partial \mathcal B^{2}=\Sph^{1}$ is $O$.

 In polar coordinates, a fixed  inward geodesic of $\mathcal B^{2}$, namely a ray
 $\rho$ through the origin $O$, is given by
$\rho(t)=(1-r(t),\theta_{0})$
where $r: [0,t_0]\to [0,1]$ are the usual rays from the origin, $\theta_{0}$ is a constant, namely the angle of the ray with $Ox$ axis. The tangent vector to the inward ray is $\dot \rho=-(\dot r(t),0)$ and the ray $\rho(t)$
satisfies the geodesic equation for the flat Euclidean metric 
\begin{equation}\label{Eucludean in polar coordinates}
a_{ij}=\begin{pmatrix}
1 & 0 \\
0 & r^2
\end{pmatrix}.
\end{equation}

The Riemannian length of the tangent vector $\dot\rho$ is $\alpha(\dot\rho(t))=\dot r(t)$, and the Riemannian length of the ray is
\begin{equation}\label{Riem length}
\int_0^{t_0}\alpha(\rho(t),\dot\rho(t))dt=\int_0^{t_0}\dot r(t)dt=r(t_0)-r(0)=1-0=1
\end{equation}
as expected.

We are going to construct a magnetic field $\beta=b_1(r,\theta)dr+
b_2(r,\theta)d\theta$ acting along the inward rays $\alpha$-orthogonal to the boundary $\partial \mathcal B^{2}$. Moreover, we will ask for this magnetic field to vanish at both ends of the inward rays. 

We start with a Lemma.

\begin{lemma}\label{lemma for h}
There exists an even, non constant $C^{\infty}$ function $h:[-1,1]\to [0,1)$ such that $h(-1)=0=h(1)$.
\end{lemma}
\begin{proof}
Choose any constants $c\in (0,30.05)$ and $\delta\in (0,1)$. 

The function 

\begin{equation}
h(x)=
\begin{cases}
0, & -1\leq t\leq -\delta\\
cg(\frac{t}{\delta}+2)g(-\frac{t}{\delta}+2), & -\delta<t<\delta\\
0, & \delta\leq t\leq 1
\end{cases}
\end{equation}
has the desired properties, where $g$ is the function 
\begin{equation}
g:(-3,3)\setminus \{0,1]\to [0,1],\qquad g(t):=\frac{\varphi(t)}{\varphi(t)+\varphi(1-t)}.
\end{equation}

Remark that the maximum of $h$ is $h(0)=0.033c<1$, for $c$ chosen as above.

 
\begin{center}
\includegraphics[height=7cm, angle=0]{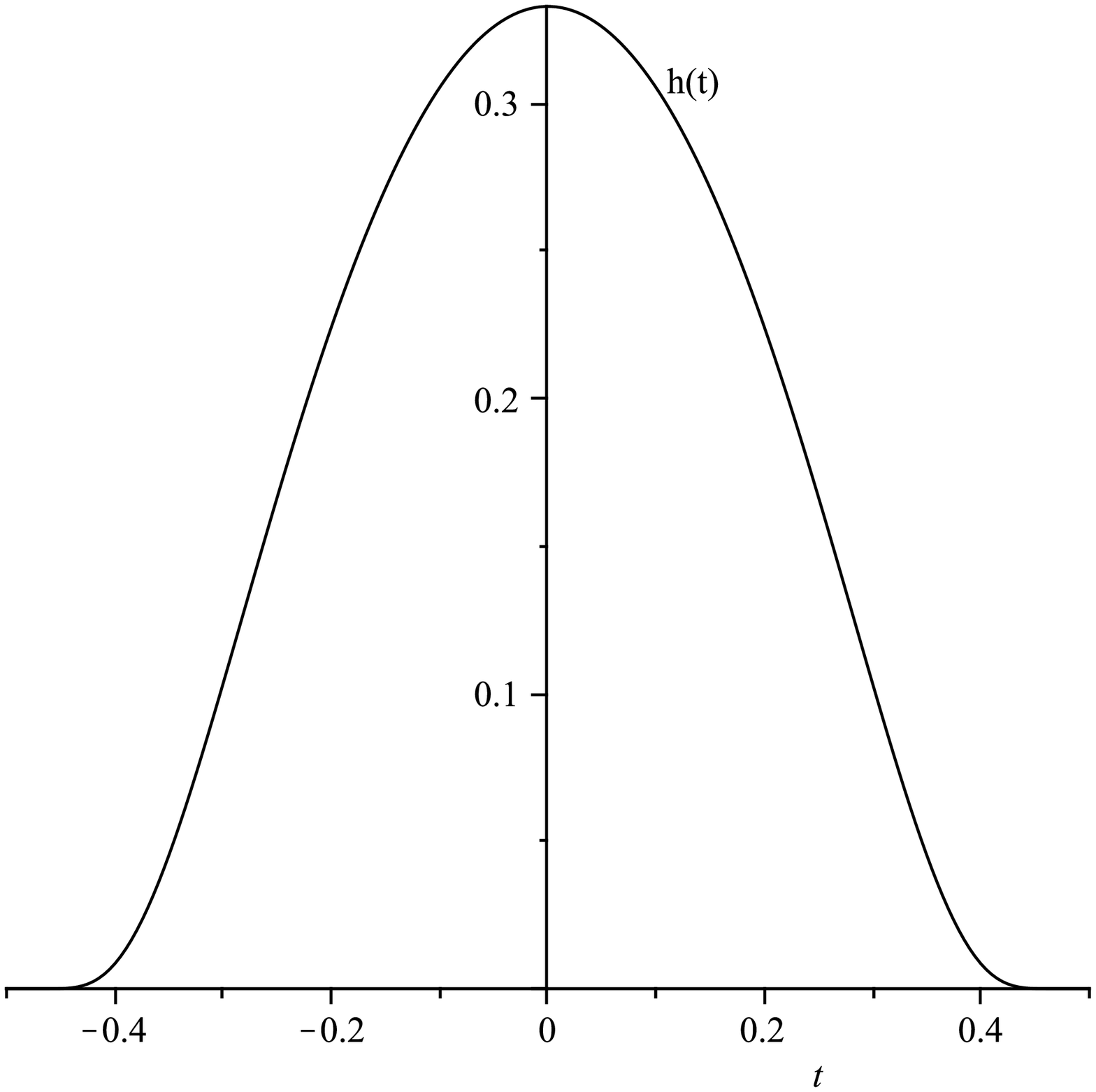}
\end{center}

{\bf Figure 11.} The graph of function $h$ for $c=10$ and $\delta=\frac{1}{2}$.

$\qedd$
\end{proof}

\begin{remark}\label{rem for h}
The function $h$ in Lemma \ref{lemma for h} can be actually defined on any finite interval $[a,b]\subset \R$ and extended to a function  $h:[a,b]\to [0,1)$ with $h(a)=h(b)=0$. 
\end{remark}

We consider the magnetic field defined by the 1-form $\beta=b_1(r,\theta)dr+
b_2(r,\theta)d\theta$ given by
\begin{equation}\label{simple magnetic field}
b_1=h(r),\qquad b_2=0
\end{equation}
where $h:[0,1]\to [0,1)$ is a function constructed as shown in Lemma \ref{lemma for h}.

Then we have
\begin{proposition}\label{Randers properties}
\begin{enumerate}
\item
The fundamental function $F=\alpha+\beta$, where $\alpha$ is given by \eqref{Eucludean in polar coordinates} and $\beta=h(r)dr$, is a positive definite  
Randers metric on $\mathcal B^{2}$. 
\item 
The Randers metric $(\mathcal B^{2},F)$ is projectively flat.
\item
The cut locus of the boundary $\partial \mathcal B^{2}$ with respect to the Randers metric is the origin $O$.
\end{enumerate} 
\end{proposition}
\begin{proof}\begin{enumerate}
\item
Remark that $b=h(r)<1$ and therefore $F$ is a positive definite Randers metric. 

\item Moreover, the 1-form $\beta$ is 
closed by construction, therefore the Randers geodesics coincide with the underlying Riemannian geodesics as set of points and the second statement follows. 

\item For the last statement, 
we remark that near the center $o$ and the boundary $\partial \mathcal B^{2}$ the magnetic field is zero, therefore the geodesic rays $\rho$ are orthogonal to $\partial \mathcal B^{2}$ with respect to the Randers metric $F$ as well.

Taking into account that the Riemannian length of the inward geodesic rays $\rho$ from $\partial \mathcal B^{2}$ with respect to the Riemannian metric $a$ is the same, then 
it can be seen that all inward geodesic rays emanating orthogonal to the boundary $\partial \mathcal B^{2}$ gather in the origin $O$ and they have the same Finslerian length. 

Indeed, let $\rho:[0,t_0]\to \R^2$ a ray through the origin $o\in \R^2$. Then the Finslerian length of the ray $\rho$ is 
\begin{equation}\label{finslerian length}
\mathcal L_F(\rho_{|[0,t_0]})=1+\int_0^1 h(\rho)d\rho
\end{equation}
and it depends only on the Riemannian length of the ray $\rho$.

In order to see this remark that
\begin{equation}
\beta(\rho(t),\dot\rho(t))=-h(1-r(t))\dot r(t).
\end{equation}

Then the integrals
\eqref{Riem length}
and 
\begin{equation}
\int_0^{t_0}\beta(\rho(t),\dot\rho(t))dt=-\int_0^{t_0}f(r)\dot r(t)dt=-\int_0^{r(t_0)}h(r)dr=-\int_0^{l}h(r)dr
\end{equation}
together with the relation 
\begin{equation}
\mathcal L_F(\rho)=\int_0^{t_0}(\alpha(\rho(t),\dot\rho(t))+\beta(\rho(t),\dot\rho(t)))dt
\end{equation}
imply \eqref{finslerian length}.
In other words, the cut locus of the boundary with respect to the Randers metric $F$ is $O$.

\end{enumerate}

$\qedd$
\end{proof}

 This construction can be easily extended to the case of an $n$-dimensional ball $\mathcal B^{n}\in \R^n$ with the Euclidean metric $a$. In this case, 
we can construct by the same procedure as above a Randers metric $(\mathcal B^{n},F=\alpha+\beta)$, whose inward geodesics coincide to the geodesic rays from the boundary $\partial \mathcal B^{n}=\Sph^{n-1}\subset\R^n$ and whose inward cut locus is the center of the sphere $\Sph^{n-1}$.


\section{Proof of Theorem \ref{theorem B}}

%

\quad  As explained already the inward cut locus from the $C^1$-boundary $\partial D_0$ with respect to the usual Euclidean metric in $\R^2$ coincides with the segment $oq$. We are going to construct a Randers metric on the 2-ball $D_0$ whose inward cut locus of the boundary $\partial D_0$ is the same segment $oq$.

\setlength{\unitlength}{1cm} 
\begin{center}
\begin{picture}(8, 7)

\includegraphics[height=7cm, angle=0]{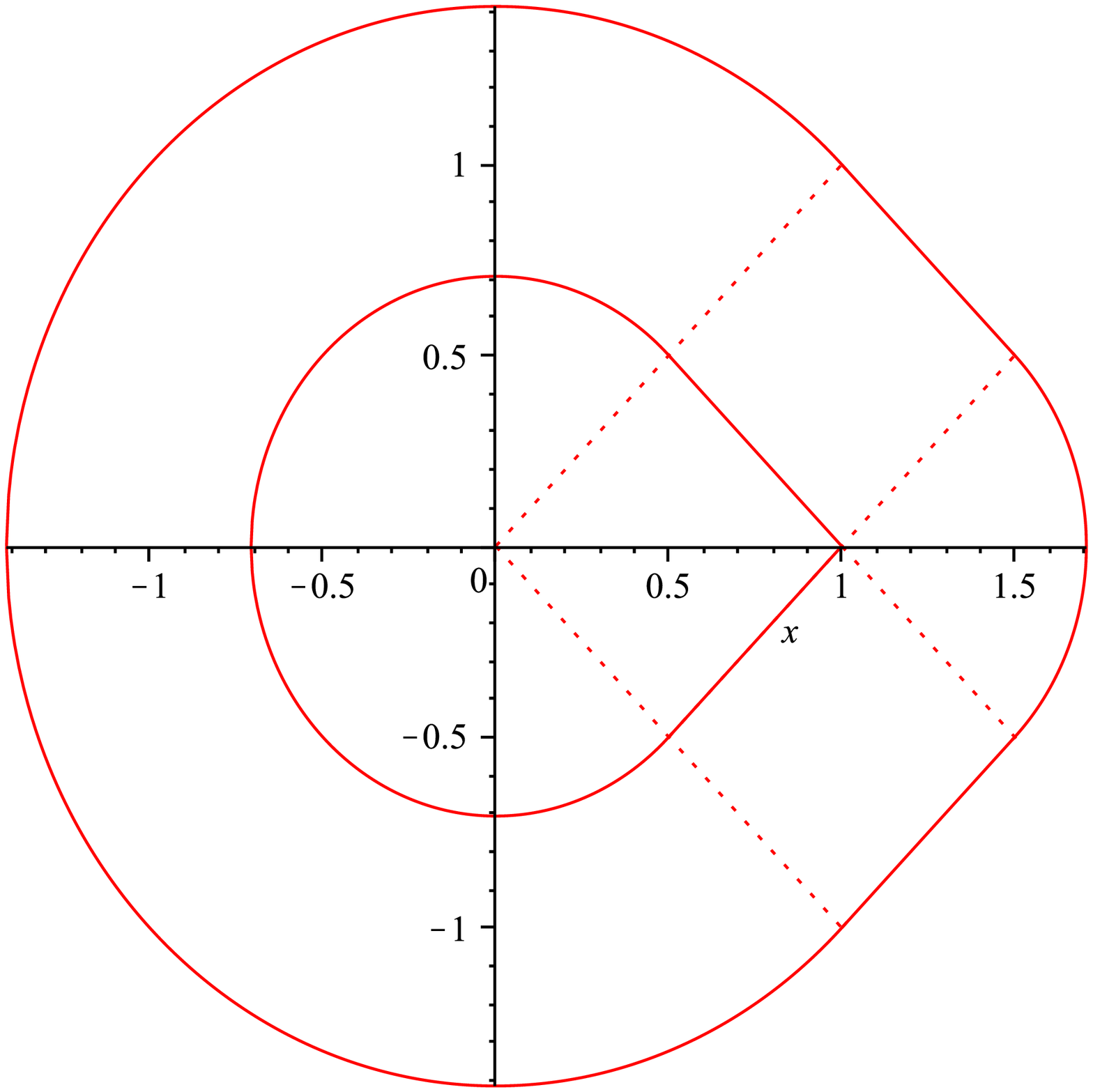}
\put(-5.7,5.7){\vector(1,-1){0.7}}
\put(-1.8,5.7){\vector(-1,-1){0.7}}
\put(-0.7,4.6){\vector(-1,-1){0.7}}
\put(-5.8,1.6){\vector(1,1){0.7}}
\put(-1.9,1.3){\vector(-1,1){0.7}}
\put(-6.4,3.6){\vector(1,0){0.7}}
\put(-0.5,3.6){\vector(-1,0){0.7}}
\put(-3.7,6.2){\vector(0,-1){0.7}}
\put(-3.7,0.7){\vector(0,1){0.7}}
\put(-1.2,5.1){\vector(-1,-1){0.7}}
\put(-0.6,2.6){\vector(-1,1){0.7}}
\put(-1.2,2){\vector(-1,1){0.7}}
\end{picture}
\end{center}

{\bf Figure 12.} The magnetic field defined within the convex ball $H_0$.
\bigskip

In order to write our formulas explicitly in polar coordinates $(r,\theta)$, for the sake of clarity and simplicity, we will consider the circles $S$ and $S'$ to be centred at origin with radius $\sqrt{2}$, and at the point $A(1,0)$ with radius $\frac{\sqrt{2}}{2}$, respectively. We also take $\phi=\frac{\pi}{4}$.

Remark that the 2-ball $D_0$ is made of three regions (compare with Figure 2 and notations in Section 2):
\begin{enumerate}[(1)]
\item $(x,y)=(r_{1}\cos\theta,r_{1}\sin\theta)\in B(o,\sqrt{2})$, where $r_{1}\in [0,\sqrt{2}],\ \theta\in [\frac{\pi}{4},\frac{7\pi}{4}]$, i.e. the large circular sector whose boundary is $P$;
\item $(x,y)=(u+\dfrac{\sqrt{2}}{2}\rho+\dfrac{1-u}{2},\pm\dfrac{\sqrt{2}}{2}\rho\pm\dfrac{1-u}{2})$,
where $\rho\in [0, \dfrac{\sqrt{2}}{2}]$, $u\in[0,1]$, i.e.  a part of the central region whose boundary is $A$;
\item  $(x,y)=(r_{2}\cos\theta,r_{2}\sin\theta)\in B((1,0),\frac{\sqrt{2}}{2})$, where 
$r_{2}\in [0,\frac{\sqrt{2}}{2}],\ \theta\in [-\frac{\pi}{4},\frac{\pi}{4}]$, i.e. the small circular sector whose boundary is $P_0$.
\end{enumerate}

We point out that $(r_{1},\theta)$ and $(r_{2},\theta)$ are polar coordinates in the balls $B(o,\sqrt{2})$ and 
$B((1,0),\frac{\sqrt{2}}{2})$, while $(\rho,u)$ are coordinates on the interior of the squares defined by the points $(\dfrac{1}{2},\dfrac{1}{2})$, $(1,1)$, $(\dfrac{3}{2},\dfrac{1}{2})$, $(1,0)$ and 
$(\dfrac{1}{2},-\dfrac{1}{2})$, $(1,-1)$, $(\dfrac{3}{2},-\dfrac{1}{2})$, $(1,0)$, respectively. Nevertheless, we have $r_{1}=\rho+\frac{\sqrt{2}}{2}$.

We are going to define a magnetic field $\beta$ acting on the (inner) straight rays from the boundaries of each regions as follows.

\begin{equation}
\beta=
\begin{cases}
\beta_{1}:=-h_{1}(r_{1})dr_{1},\quad x\in \textrm{ region (1)}\\
\beta_{2}:=-h_{2}(\rho)d\rho,\quad x\in \textrm{ region (2)}\\
\beta_{3}:=-h_{3}(r_{2})d r_{2},\quad x\in \textrm{ region (3)}
\end{cases}
\end{equation}
defined on $D_{0}$, where $h_{1}:[\frac{\sqrt{2}}{2},\sqrt{2}]\to [0,1)$, $h_{2}:[0,\frac{\sqrt{2}}{2}]\to [0,1)$ and $h_{3}:[0,\frac{\sqrt{2}}{2}]\to [0,1)$ are constructed as shown in Lemma \ref{lemma for h} and Remark \ref{rem for h}.

\setlength{\unitlength}{1cm} 
\begin{center}
\begin{picture}(7, 9)

\includegraphics[height=9cm, angle=0]{dilatation.eps}
\put(-5.52,8){\vector(0,1){0.1}}
\put(-0.3,4.3){\vector(1,0){0.1}}

\put(-7.6,6.7){\vector(1,-1){1}}
\put(-7.6,2.2){\vector(1,1){1}}

\put(-2.8,6){\vector(-1,-1){1}}
\put(-2.8,2.4){\vector(-1,1){1}}

\put(-1.9,3.5){\vector(-2,1){1}}
\put(-1.9,5){\vector(-2,-1){1}}






\put(-8.6,4.5){$\varepsilon$}



\end{picture}
\end{center}

{\bf Figure 13.} The magnetic field extended in $\widetilde{D}_0$.
\bigskip

By naturally extending $\beta$ to the ball $\widetilde{D}_{0}$ with smooth boundary, then it can be easily seen that the Randers metric $(\widetilde{D}_{0},F_{0}=\alpha+\beta_{0})$ is projectively flat and the cut locus of the hypersurface $\partial \widetilde{D}_{0}$ with respect to this Finsler metric is the segment $oq$. This can be easily done by extending the definition domain of the functions $h_{i}$ such that $h_{|\partial \widetilde{D}_{0}}=0$, namely $h_{1}:[\frac{\sqrt{2}}{2},\sqrt{2}+\varepsilon]\to [0,1)$, $h_{2}:[0,\frac{\sqrt{2}}{2}+\varepsilon]\to [0,1)$ and $h_{3}:[0,\frac{\sqrt{2}}{2}+\varepsilon]\to [0,1)$, respectively.

This construction can be extended to arbitrary dimension and repeated iteratively for each depth level $m$ such that we obtain
\begin{proposition}
For each depth level $m$, there is a Randers metric $(\widetilde{D}_{m},F_{m}=\alpha+\beta_{m})$ on the smooth ball $\widetilde{D}_{m}\subset \R^{n}$ with the following properties
\begin{enumerate}
\item $F_{m}$ is projectively flat,
\item the inner cut locus of the boundary $\partial \widetilde{D}_{m}$ with respect to $F_{m}$ is the tree 
$s\cup s_{j_{1}}\cup \dots\cup s_{j_{1}j_{2}\dots j_{m}}$.
\end{enumerate}
\end{proposition}

This result is interesting in itself because it gives a simple example of Randers metric whose cut locus of a closed hypersurface is a tree (compare \cite{TS}).

At limit we obtain
\begin{proposition}
There is a Randers metric $(\widetilde{D},F=\alpha+\beta)$ on the smooth ball $\widetilde{D}=\lim_{m\to \infty}\widetilde{D}_{m}\subset \R^{n}$ with the following properties
\begin{enumerate}
\item $F$ is projectively flat,
\item the inner cut locus of the boundary $\partial \widetilde{D}$ with respect to $F$ is the infinite tree IT,
\end{enumerate}
where $F=\lim_{m\to\infty}F_{m}$.
\end{proposition}

We remark that, for any finite $m$, the magnetic field $\beta_{m}$ defined here is initially defined inside regions (1), (2), (3) and extended by $\varepsilon$-dilatation. However, one can easily see that the regions (2) and (3), considered inside $H_m$, become smaller as $m$ increases. 
At limit $m\to \infty$, the regions (2), (3) shrink to domains of Hausdorff dimension one and zero, respectively, in other words the intensity of magnetic field $\beta$ inside $H_m$ decreases to zero as $m$ approaches infinity. Nevertheless, $\beta$ is unchanged in region (1) and in the $\varepsilon$-dilatation of (2) and (3). 


\begin{proof}[Proof of Theorem \ref{theorem B}]

Let us consider again the construction of the $n$-sphere $M=\Sph^{n}$ endowed with a Riemannian metric $g$ such that $g$ restricted to $M\setminus (interior\ of\ \Delta)$ coincides with with the initial metric $a_{ij}$ defined on $\widetilde D$ (see the Proof of Theorem \ref{theorem A} in Section \ref{sec:4}). 

On the other hand, we have the magnetic field $\beta$ defined above on $\widetilde D$. For the sake of simlicity we assume here that  $\beta$ actually acts only on the $\varepsilon$-dilatation region bounded by $\widetilde H$ and $\partial\widetilde D$. 

From the discussion above it follows that the Randers metric $(M,g+\beta)$  is a Finsler metric on the $n$-sphere $M=\Sph^{n}$ whose cut locus is $IT$ having the same order of differentiability with $(M,g)$.

This magnetic field is acting only on the tropical region of the North Hemisphere. 
\setlength{\unitlength}{1cm} 
\begin{center}
\begin{picture}(7, 8)

\includegraphics[height=7cm, angle=0]{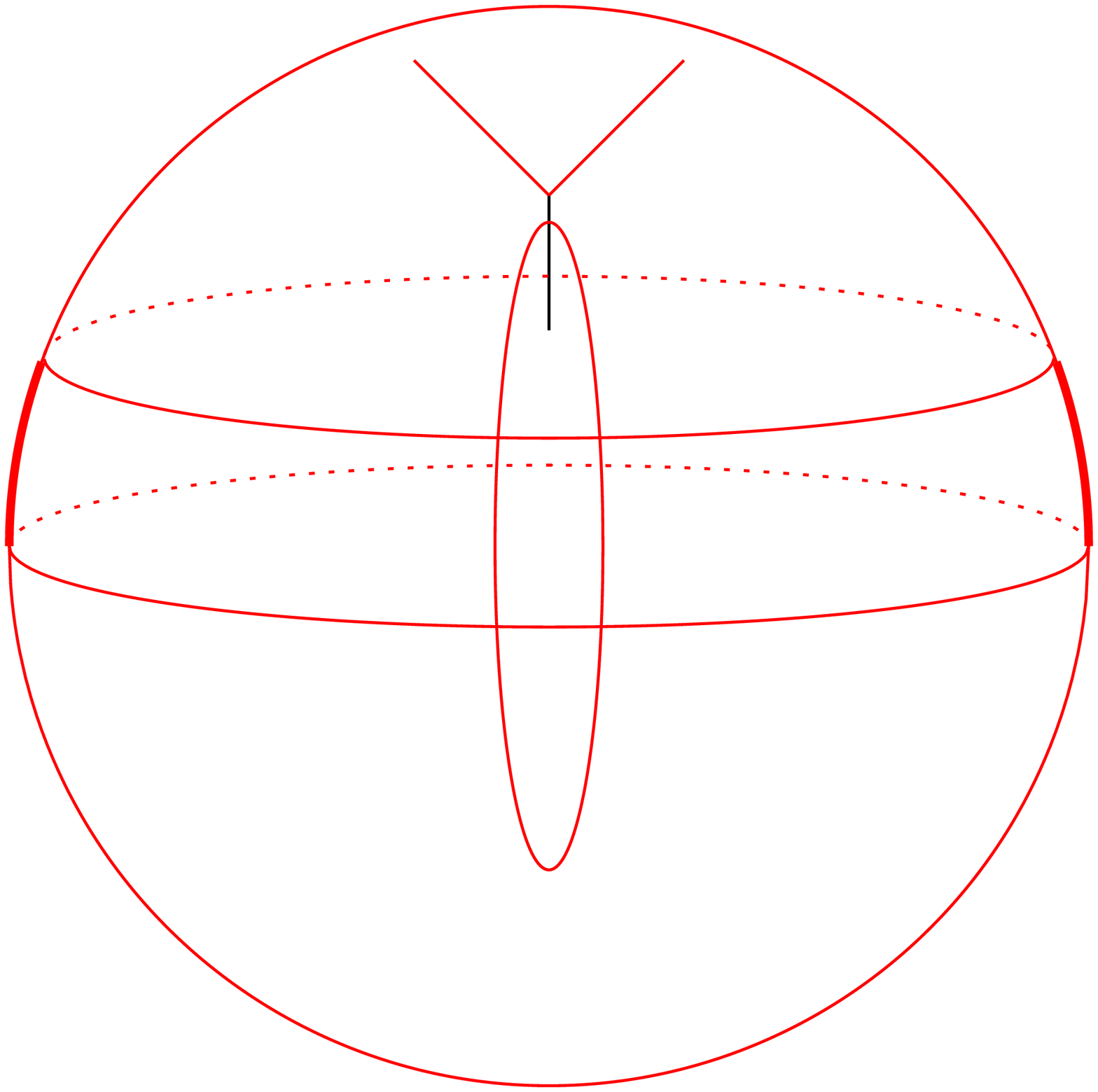}
\put(-5.7,3.4){\vector(0,1){0.7}}
\put(0.3,3.9){$         \textrm{\LARGE{\}}}       \ \ \varepsilon-dilatation$}
\put(-1,6){$ M= \Sph^{n}$}
\put(-1.7,3.4){\vector(0,1){0.7}}
\put(-4.7,3.4){\vector(0,1){0.7}}
\put(-1.3,3.6){$\beta$}
\put(-2.65,3.4){\vector(0,1){0.7}}
\put(-3.7,1.2){$p$}
\put(-3.55,1.45){$.$}
\put(-3.3,3.4){$\wedge$}
\put(-3.96,3.4){$\wedge$}
\put(-3.15,2){$\gamma$}
\put(-2.9,6){$IT$}
\put(-7.7,3.2){$\Sph^{n-1}$}
\put(-7.5,4.5){$\partial D$}
\put(-5.6,5.6){$\widetilde D$}
\put(-5.6,1.3){$\Delta$}
\end{picture}
\end{center}

{\bf Figure 14.} The Randers metric on $M$ with tropical magnetic field.
\bigskip

Of course $\beta$ is a closed 1-form and therefore the geodesics of the Randers metric $F=g+\beta$ coincide with the Riemannian geodesics of $(M,g)$ as point sets. 

The cut locus of $p$ with respect to the Randers metric $F=g+\beta$ coincides with the infinite tree $IT$ and since $(M,g)$ is $C^{k}$, but not $C^{\infty}$ this property is inherited by $F$ as well and hence Theorem \ref{theorem B} is proved.

$\qedd$
\end{proof}

\begin{remark}
Actually, 
this magnetic field $\beta$ can be extended in the interior of $\Delta$. In this way we obtain a Randers metric whose magnetic field acts in all  $\Delta$. 

 Let us denote by 
$\gamma:[0,a]\to M$, $\gamma(0)=p$, $\gamma(a)=q$ a minimizing geodesic segment of $(M,g)$ that joins the point $p$ with a point $q\in \mathcal C (p)$. Using for example Riemannian geodesic coordinates, any geodesic segment from $p$ to its cut point is a straight line. Then using Lemma \ref{lemma for h} and Remark \ref{rem for h} we can extend the tropical magnetic field defined on the $\varepsilon$-dilatation region to entire hemisphere $\Delta$.

Obviously, our magnetic field is zero at $p$, increases in strength and decreases again to zero after crossing the equator when moving from south to north such that $\beta$ vanishes on $\partial D$.

\end{remark}



\bigskip


Jinichi ITOH,

Faculty of Education,
Kumamoto University,
Kumamoto 860,
Japan

\medskip
{\tt j-itoh@kumamoto-u.ac.jp}

\bigskip

Sorin V. SABAU,


School of Science,
Dep. of Mathematics,
Tokai University,
Sapporo 005\,--\,8600,
 Japan

\medskip
{\tt sorin@tspirit.tokai-u.jp}

\end{document}